\documentclass[10 pt,a4paper]{article}

\usepackage{latexsym,amsfonts,amsmath,amsthm,amssymb, graphicx, mathrsfs, fancyhdr, makeidx, amscd,  color, mathdesign, fontenc, stix}
\usepackage[all]{xy}

\newtheorem{theorem}{Theorem}[section]

\newtheorem{corollary}[theorem]{Corollary}

\newtheorem{definition}[theorem]{Definition}
\newtheorem{example}[theorem]{Example}

\newtheorem{lemma}[theorem]{Lemma}

\newtheorem{proposition}[theorem]{Proposition}
\newtheorem{remark}[theorem]{Remark}

\newcommand{\R}{\mathbb{R}}

\newcommand{\disp}{\displaystyle}

\newcommand{\ra}{\rightarrow}

\newcommand{\eps}{\varepsilon}

\newcommand{\di}{\mathrm{d}}

\newcommand{\lip}{\mathrm{Lip}}
\newcommand{\loc}{\mathrm{loc}}

\newcommand{\USC}{\mathrm{USC}}
\newcommand{\LSC}{\mathrm{LSC}}
\newcommand{\capac}{\mathrm{cap}}

\DeclareMathOperator{\diver}{div\,}

\DeclareMathOperator{\Hess}{\mathrm{Hess}}
\DeclareMathOperator{\hess}{\mathrm{Hess}}

\begin{document}

\author{Dami\~ao J. Ara\'ujo \and Luciano Mari \and Leandro F. Pessoa}
\title{\textbf{Detecting the completeness of a Finsler manifold via potential theory for its infinity Laplacian}}
\date{}
\maketitle
\scriptsize \begin{center} Departamento de Matem\'atica, Universidade Federal da Para\'iba\\
58059-900, Jo\~ao Pessoa - Para\'iba (Brazil)\\
E-mail: araujo@mat.ufpb.br
\end{center}

\scriptsize \begin{center} Dipartimento di Matematica, Universit\`a degli studi di Torino,\\
Via Carlo Alberto 10, 10123 Torino (Italy)\\
E-mail: luciano.mari@unito.it
\end{center}

\scriptsize \begin{center} Departamento de Matem\'{a}tica, Universidade Federal do Piau\'{i},\\
64049-550, Teresina (Brazil)\\
E-mail: leandropessoa@ufpi.edu.br
\end{center}

\normalsize

\maketitle

\begin{abstract}
\noindent In this paper, we study some potential theoretic aspects of the eikonal and infinity Laplace operator on a Finsler manifold $M$. Our main result shows that the forward completeness of $M$ can be detected in terms of Liouville properties and maximum principles at infinity for subsolutions of suitable inequalities, including $\Delta^N_\infty u \ge g(u)$. Also, an $\infty$-capacity criterion and a viscosity version of Ekeland principle are proved to be equivalent to the forward completeness of $M$. Part of the proof hinges on a new boundary-to-interior Lipschitz estimate for solutions of $\Delta^N_\infty u = g(u)$ on relatively compact sets, that implies a uniform Lipschitz estimate for certain entire, bounded solutions without requiring the completeness of $M$.
\end{abstract}


\tableofcontents

\newpage 

\section{Introduction}



This work is about a potential theory for the $\infty$-Laplace operator 
	\[
	\Delta_\infty u : = \hess u(\nabla u, \nabla u)
	\] 
and its normalized version
	\[
	\Delta_\infty^N u : = \hess u \left( \frac{\nabla u}{|\nabla u|}, \frac{\nabla u}{|\nabla u|} \right)
	\]	
on a Finsler manifold. The infinity Laplacian has received great attention after the pioneering work of G. Arronson \cite{aronsson1,aronsson2} in the 1960s, and showed intriguing connections with pure and applied mathematical issues, as for example, Tug-of-war games \cite{Bar_Evans_Jensen, peres_schramm_sheffield_wilson, rossi_survey}, mass transportation problems \cite{transportmass} and others. The study of the infinity Laplacian is strictly related with an $L^\infty$ minimization problem: given a bounded domain $\Omega \subset \mathbb{R}^m$ and a Lipschitz function $\zeta: \partial \Omega \to \mathbb{R}$, to find an extension $u$ of $\zeta$ in $\Omega$ such that the Lipschitz constant $\lip(u,A) \leq \lip(h,A)$ for any $A \Subset \Omega$ and $h$ which agrees with $u$ on $\partial A$. Such function is called an absolutely minimizing Lipschitz  extension, shortly AMLE \cite{crandall_visit, champion_depascale}. Jensen in \cite{jensen} showed that the AMLE property is equivalent to the fact that $u$ be a viscosity solution for $\Delta_\infty u = 0$, and by \cite{crandall_evans_gariepy,jensen} AMLEs are also characterized by the comparison principle with cone functions 
$$
C_x(y)= a + b |x-y| \quad a,b \in \mathbb{R} \mbox{ and } y \in \mathbb{R}^m,
$$
which are fundamental solutions of the homogeneous infinity Laplacian. This is the tripod that supports the role of the basic theory of infinity harmonic functions on, say, $\R^m$ with its standard metric. Since then, various works have been devoted to the analysis of $\Delta_\infty$ for more general structures, and an account can be found in \cite{crandall_visit,aronsson_crandall_juutinen}. Especially, on domains of $\R^m$ equipped with a Finsler norm, the AMLE problem and the associated $\infty$-Laplace operator have been studied in \cite{wu,GXY,mebrate_mohammed,mebrate_mohammed_2}.\par
One of the starting points of the present investigation is the following Liouville theorem for $\infty$-subharmonic functions on $\R^m$ (cf. \cite{lindqvistmanfredi,crandall_evans_gariepy}):
	\begin{equation}\label{teo_liouRm}
	\emph{entire viscosity solutions of } \, \Delta_\infty u \ge 0 \ \emph{ with $\sup_{\R^m} u < \infty$  are constant.}
	\end{equation}
Its proof is a consequence, for instance, of the Harnack inequality for $\infty$-subharmonic equations \cite{lindqvistmanfredi,lindqvistmanfredi_2,juutinen_2} (cf. also \cite{evans}):
	\begin{equation}\label{eq_Harnack}
	u(x)-\sup_M u \le \big[u(y)-\sup_M u\big]e^{-\frac{|x-y|}{R-r}} \qquad \forall \, x,y \in B_z(r), \ \ R> r,
	\end{equation}
by letting $R \ra +\infty$. It is natural to ask for which class of manifolds the above theorem remains true; inspection of the proof of \eqref{eq_Harnack} reveals that the completeness of $\R^m$ is used, and suggests that \eqref{teo_liouRm} be true for any complete Riemannian manifolds, regardless to curvature requirements. This is, we shall see, easy to prove. However, the question whether \eqref{teo_liouRm} holds \emph{only} on complete manifolds is more interesting and, to our knowledge, only studied in recent years. Our investigation arose in the context of fully nonlinear potential theory, motivated by the desire recast, in a unified framework, various maximum principles at infinity available in the literature: the celebrated Ekeland \cite{ekeland_2, ekeland} and Omori-Yau ones \cite{omori,yau,chengyau}, as well as those coming from stochastic geometry (the weak maximum principles of Pigola-Rigoli-Setti \cite{prsmemoirs}, related to parabolicity, stochastic and martingale completeness of a Riemannian manifold). This investigation initiated in \cite{maripessoa, maripessoa_2}, in a Riemannian setting, see also previous results in \cite{pigolasetti_ensaio, prs_overview}.  The need to consider first order conditions in the statements of Ekeland and Omori-Yau principles requires to include the eikonal into the class of equations to which the theory be applicable, and opened the way to also encompass the $\infty$-Laplace operator, tightly related to the eikonal one. In Theorem 1.12 of \cite{maripessoa}, the geodesic completeness of a Riemannian manifold (i.e., the completeness of $M$ as a metric space) is shown to be equivalent to various other conditions, among them a suitable version of Ekeland principle for viscosity solutions (that, consequently, turns out to be equivalent to the original Ekeland formulation), and the validity of \eqref{teo_liouRm} on $M$.\par
	In the present work, we move some step further and improve Theorem 1.12 in \cite{maripessoa} on various aspects. First, we extend the investigation from Riemannian to Finsler manifolds, where the possible asymmetry of the metric introduces further issues; we hope to convey our feeling that the Finsler setting is quite natural for the problems we study herein. Second, we also consider inhomogeneous inequalities of the type
	\[
	\Delta_\infty^N u \ge g(u)
	\] 
for continuous, non-negative $g$. The main purpose is to discover whether a Liouville property for bounded solutions of $\Delta_\infty^N u \ge g(u)$ for \emph{some} non-negative $g$ still detects the completeness of $M$ (more precisely, the forward completeness of the Finsler manifold $(M,F)$), or rather a weaker property. \par
	To comment on this point, and to motivate the conditions in the statement of our main theorem, we begin with the analogy between \eqref{teo_liouRm} and a corresponding statement for the Laplace operator on a manifold $M$: 
	\begin{equation}\label{teo_liouLapla}
	\emph{entire solutions of } \, \Delta u \ge 0 \ \emph{with $\sup_M u < \infty$ are constant},
	\end{equation}
a well-known property in potential theory that was the subject of intense investigation starting from the $2$-dimensional case, where the validity or failure of \eqref{teo_liouLapla} characterizes the conformal type of a simply connected Riemann surface. A Riemannian manifold for which \eqref{teo_liouLapla} holds is named \emph{parabolic}. As observed in \cite[Thm. 6C]{ahlforssario}, \eqref{teo_liouLapla} can equivalently be  expressed as the following maximum principle at infinity:
	\begin{equation}\label{teo_para}
	\begin{array}{l}
	\emph{for every $\Omega \subset M$ open and $u \in C(\overline\Omega)$ solving}\\[0.2cm]
	\left\{ \begin{array}{l}
	\Delta u \ge 0 \qquad \text{on } \, \Omega, \\[0.2cm]
	\sup_\Omega u < \infty \qquad \text{on } \, \Omega 
	\end{array} \right. \qquad \Longrightarrow \qquad \sup_\Omega u = \sup_{\partial \Omega} u. 
	\end{array}
	\end{equation}
Compact manifolds are clearly parabolic, so the property characterizes non-compact manifolds that are, somehow, not far from being compact. In view of applications to a variety of geometric problems (see \cite{prsmemoirs, amr}), it is useful to investigate versions of \eqref{teo_para} for inhomogeneous equations like $\Delta u \ge g(u)$, with $g \in C(\R)$. Quite interestingly, they relate to a property that, like parabolicity, ties to the theory of stochastic processes: the \emph{stochastic completeness} of $M$. Briefly, $M$ is parabolic if the minimal Brownian motion $\mathscr{B}_t$ on $M$ is recurrent, that is, almost surely, its trajectories visit any fixed compact set infinitely often along a divergent sequence of times. On the other hand, $M$ is said to be stochastically complete if $\mathscr{B}_t$ is non-explosive, that is, if trajectories of $\mathscr{B}_t$ have infinite lifetime almost surely. Note that, by their very definitions, parabolic manifolds are stochastically complete, but the viceversa is far from being true: for instance, if $M$ is geodesically complete, sufficient conditions for the parabolicity and stochastic completeness are, respectively, 
	\begin{equation}\label{eq_volgrowth}
	\int^{+\infty} \frac{s\di s}{|B_s|} = +\infty, \qquad \text{and} \qquad \int^{+\infty} \frac{s\di s}{\log |B_s|} = +\infty, 
	\end{equation}
where $|B_r|$ is the volume of a geodesic ball centered at a fixed origin. The two criteria, sharp for relevant classes of manifolds, can be found in Theorems 5.1 and 6.2 of \cite{grigoryan}, that we suggest to consult for a detailed account. While the first in \eqref{eq_volgrowth} is somehow binding (for instance, $\R^m$ is parabolic if and only if $m =2$), the volume  threshold to match the second in \eqref{eq_volgrowth} is of the order of $e^{r^2}$, and includes many more Riemannian manifolds of interest in geometry, for instance all of those with Ricci curvature bounded from below by a constant (cf. \cite{prsmemoirs, amr}). Characterizations of the stochastic completeness of $M$ in terms of maximum principles at infinity were found in \cite{prs_proceeding, prsmemoirs, amir}, and among equivalent statements we choose here the following one: to state it, the lack of translation invariance of the inequality $\Delta u \ge g(u)$ requires to fix a normalization threshold, taken to be zero for convenience. Then, the principle writes as follows:
	\begin{equation}\label{teo_WMP}
	\begin{array}{l}
	\emph{for some/every $g \in C(\R)$ with $g(0)=0$, $g>0$ on $\R^+$, the following holds:} \\
	\emph{for every $\Omega \subset M$ open and $u \in C(\overline\Omega)$ solving} \\[0.2cm]
	\left\{ \begin{array}{l}
	\Delta u \ge g(u) \qquad \text{on } \, \Omega, \\[0.2cm]
	0 < \sup_\Omega u < \infty \qquad \text{on } \, \Omega 
	\end{array} \right. \qquad \Longrightarrow  \qquad \sup_\Omega u = \sup_{\partial \Omega} u.
	\end{array}
	\end{equation}
Loosely speaking, the principle guarantees the non-existence of functions $u$ that are bounded from above and solve $\Delta u \ge g(u)$ on a non-empty upper level set $\{ u > \gamma\}$, for some $\gamma \ge 0$. Solutions of $\Delta u \ge g(u)$ can be taken in either the viscosity or the weak sense.\par
	Geometric applications also motivated the study of maximum principles at infinity when the Laplacian is replaced by more general, nonlinear operators, notably including the mean curvature one 
	\[
	\diver \left( \frac{\nabla u}{\sqrt{1 + |\nabla u|^2}} \right) 
	\]
and the $p$-Laplacian
	\[
	\Delta_p u \doteq \diver (|\nabla u|^{p-2} \nabla u), \qquad p \in (1, \infty). 
	\]
For instance, the first operator appears when studying entire graphs with prescribed mean curvature, and the validity of maximum principles at infinity are therefore instrumental to prove Bernstein type theorems \cite{bmpr,cmmr}, while the $p$-Laplacian, in the limit $p \ra 1$, gives an efficient way to construct solutions of the inverse mean curvature flow on spaces with mild curvature requirements, see  \cite{maririgolisetti_mono}, and maximum principles at infinity serve to guarantee the global gradient estimates needed to perform the approximation procedure. For both operators, criteria in the spirit of \eqref{eq_volgrowth} have been established in \cite{prsmemoirs,bmpr}, still showing a substantial difference between the ``parabolic" case $g \equiv 0$ and the case $g>0$ on $\R^+$. More precisely, the formal limit
	\[
	\Delta_\infty^N u = \lim_{p \ra \infty} \frac{|\nabla u|^{2-p}}{p} \Delta_p u
	\]	
relates solutions of the normalized equation $\Delta^N_\infty u \ge g(u)$ to those of 
	\begin{equation}\label{solu_deltap}
	\Delta_p u \ge p g(u)|\nabla u|^{p-2}
	\end{equation}
for large $p$. By Theorem 2.24 and Proposition 7.4 in \cite{bmpr}, property \eqref{teo_WMP} for solutions of \eqref{solu_deltap} holds on any complete manifold $M$ satisfying  
	\[
	\liminf_{r \ra \infty} \frac{\log |B_r|}{r^2} < \infty,
	\]
a bound that is sharp and, perhaps surprisingly, independent of $p$ (cf. Section 7.4 in \cite{bmpr}), while if $g \equiv 0$ a sharp threshold is given by 
	\[
	\int^{+\infty} \left(\frac{s \di s}{|B_s|}\right)^{\frac{1}{p-1}} = \infty,
	\]
cf. \cite{prsmemoirs} and the references therein. It is therefore tempting to wonder whether, in the limit $p \ra \infty$, the two possibilities for $g$ still detect different properties.\par 
\vspace{0.2cm}
Let $(M,F)$ be a Finsler manifold (the basics of Finsler Geometry are recalled in Section \ref{sec_prelim}). We assume the Finsler norm $F : TM \ra [0, \infty)$ be positively homogeneous of degree $1$, and $F^2$ be strictly convex when restricted on each fiber of $TM \ra M$. For smooth $u$, the Chern connection associated to $F$ allows to define the Hessian of a function and, consequently, a Finsler $\infty$-Laplacian. Also, the norm $F$ induces a pseudo-distance $\di$ on $M$ that is, $\di$ satisfies all of the requirements of a distance function but, possibly, its symmetry. The lack of symmetry introduces further issues, among them the need to distinguish which properties relate to the \emph{forward} completeness of $M$ rather than to its \emph{backward} one. The forward completeness for $(M,F)$ is defined by asking that forward Cauchy sequences converge, i.e. if $\{x_i\}$ satisfies the following Cauchy condition:
	\[
	\forall \, \varepsilon>0, \ \exists N= N(\varepsilon) \in \mathbb{N} \ : \ N\leq i < j \Longrightarrow \di(x_i,x_j) < \varepsilon,
	\]
then $\{x_i\}$ converges. Following \cite{champion_depascale}, we define the Lipschitz constant of $u$ on a set $A$ to be 
\begin{eqnarray}\label{lipschitz_eq}
\textrm{Lip}(u,A) \doteq \inf \Big\{ L \in [0,\infty] \ : \ u(y) - u(x) \leq L\di(x,y) \quad \forall\, x,y \in A\Big\}.
\end{eqnarray}
Let $\varrho^+(x) = \di(o,x)$ denotes the distance from a fixed origin $o \in M$. We are ready to state our main result. Note that solutions are meant to be in the viscosity sense, see \cite{CIL}.


\begin{theorem}\label{teo_main}
Let $(M,F)$ be a connected Finsler manifold. Then, the following properties are equivalent:
\begin{itemize}
\item[1)] $(M,F)$ is forward complete.
\item[2)] Having denoted with $\varrho^+$ the forward distance from a fixed origin,  
	\begin{equation}\label{eq_it2}
	\left\{ \begin{array}{l}
	\Delta_\infty^N u \ge 0 \qquad \text{on } \,  M, \\[0.2cm]
	u_+(x) = o\big( \varrho^+(x)\big) \quad \text{as $\varrho^+(x) \ra +\infty$} 
	\end{array}\right. \qquad \Longrightarrow \qquad \text{$u$ is constant}.
	\end{equation}
\item[3)] For some/every $g \in C(\R)$ with $g(0)=0$ and $g \ge 0$ on $\R^+$, the following holds: 
	\[
	\left\{ \begin{array}{l}
	\Delta_\infty^N u \ge g(u) \qquad \text{on } \, M, \\[0.2cm] 
	0 < \sup_M u < + \infty
	\end{array}\right. \qquad \Longrightarrow \qquad \text{$u$ is constant}.
	\]
\item[4)] For some/every $g \in C(\R)$ with $g(0) =0$ and $g\ge 0$ on $\R^+$, the following holds: for every open subset $\Omega \subset M$,  	
	\begin{equation}\label{WMP_g}
	\left\{ \begin{array}{l}
	\Delta_\infty^N u \ge g(u) \qquad \text{on } \, \Omega, \\[0.2cm] 
	0 < \sup_\Omega u < + \infty
	\end{array}\right. \qquad \Longrightarrow \qquad \sup_\Omega u = \sup_{\partial \Omega} u. 
	\end{equation}
\item[5)] For some/every $\theta \in (0,1)$ and $\lambda>0$, it holds	
	\begin{equation}\label{eq_theta}
	\left\{ \begin{array}{l}
	\Delta_\infty^N u \ge \lambda u_+^\theta \qquad \text{on } \,  M, \\[0.4cm]
	\disp \limsup_{\varrho^+(x) \ra +\infty} \frac{u_+(x)}{\varrho^+(x)^{\frac{2}{1-\theta}}} < \sqrt[1-\theta]{\lambda\frac{(1-\theta)^2}{2(1+\theta)}}  
	\end{array}\right. \quad \Longrightarrow \qquad \text{$u$ is a (nonpositive) constant}.
	\end{equation}
\item[6)] For some/every $K \subset M$ compact, it holds
	\[
	\inf_{u \in \mathscr{L}(K,M)}  \lip(u,M) = 0, 
	\]
where
	\begin{equation}\label{def_classeL}
	\mathscr{L}(K,M) = \Big\{u \in \lip_c(M), \ u \le -1 \ \text{ on } \, K \Big\}.
	\end{equation}
\item[7)] For some/every $K \subset M$ compact, the $\infty$-capacity of $K$ vanishes:
	\[
	\capac_\infty(K) : = \inf_{u \in \mathscr{L}(K,M)} \| F(\nabla u)\|_{L^\infty(M)} = 0, 
	\]
where $\mathscr{L}(K,M)$ is defined in \eqref{def_classeL}.	
\item[8)] For some/every $0 < G \in C(\R)$, the following holds: for every open subset $\Omega \subset M$, and for every viscosity subsolution of 	
	\begin{equation}\label{eq_bonita}
	\left\{ \begin{array}{l}
	G(u) - F(\nabla u) = 0 \qquad \text{on } \, \Omega, \\[0.2cm] 
	\sup_\Omega u <  \infty
	\end{array}\right. \qquad \Longrightarrow \qquad \sup_\Omega u = \sup_{\partial \Omega} u. 
	\end{equation}
\item[9)] (\textbf{Ekeland principle}). For every $u \in \USC(M)$ with $\sup_M u < \infty$, for every $\eps>0$ and $x_0 \in M$ such that $u(x_0) > \sup_M u - \eps$, and for every $\delta>0$, there exists $\bar x \in M$ such that
	\[
	\begin{array}{l}
	u(\bar x) \ge u(x_0), \quad \, \di(x_0,\bar x) \le \delta, \quad \, \text{and} \quad \,	u(y) \le u(\bar x) + \frac{\eps}{\delta}\di(\bar x,y) \quad \forall \, y \in M.
	\end{array}
	\] 
\end{itemize}
\end{theorem}

%

\begin{remark}[\textbf{The some/every alternative}]
\emph{Property $3)$, as well as $4)$, holds for \emph{every} $g$ as in the statement provided that it holds for \emph{some} such $g$. In particular, in view of our assumption on $g$, the \emph{every} alternative is equivalent to require $3)$ for the smallest choice $g \equiv 0$. Therefore, unlikely the case of $\Delta_p$ with $p < \infty$, for the $\infty$-Laplacian the Liouville theorems for $\Delta_\infty^N u \ge g(u)$ under the assumptions $g\equiv 0$ or $g(0)=0$, $g>0$ on $\R^+$ are equivalent.
}
\end{remark}

\begin{remark}[\textbf{Backward completeness}]
\emph{The notion of backward completeness for $(M,F)$, demanding that backward Cauchy sequences converge,  corresponds to the forward completeness of the dual Finsler structure
	\[
	\widetilde{F}(p) : = F(-p), \qquad  p \in TM,
	\]
hence it can be described via the eikonal and normalized $\infty$-Laplacian $\widetilde{\Delta}^N_\infty$  associated to $\widetilde{F}$. In view of the identity 
	\[
	\widetilde{\Delta}_\infty^N u = - \Delta_\infty^N ( -u), 
	\]
the backward completeness of $(M,F)$ can be detected by minimum principles for solutions of $\Delta_\infty^N u \le g(u)$. We leave the statement to the interested reader.
}
\end{remark}

\begin{remark}[\textbf{On conditions $\bf 8),9)$: a viscosity Ekeland principle}]
\emph{Implication $1) \Rightarrow 9)$ is the celebrated Ekeland principle \cite{ekeland, ekeland_2}, originally stated for metric spaces, while $9) \Rightarrow 1)$ has been pointed out by J.D. Weston \cite{weston} and F. Sullivan \cite{sullivan}. Extension to the Finsler setting is straightforward, since Weston-Sullivan arguments as well as the proof of $9)$ provided in \cite[p.444]{ekeland} do not use the symmetry of $\di$ at any stage. We included $9)$ for the sake of completeness, and to emphasize that $8)$ can be interpreted as a viscosity version of Ekeland principle. 
}
\end{remark}

\begin{remark}[\textbf{On condition $\bf 5)$}]
\emph{Reaction-diffusion equations with strong absorption as in $5)$ were investigated in \cite{araujo_leitao_teixeira}, where the authors proved regularity for the unnormalized case $\Delta_\infty u = \lambda u_+^\gamma$ in $\mathbb{R}^m$, $0\leq \gamma<3$, and related Liouville theorems for entire solutions satisfying 
\begin{equation}\label{gammacond}
u(x)=O(|x|^{\frac{4}{3-\gamma}}) \quad \mbox{as} \quad |x| \to \infty.
\end{equation}
In the limit $\gamma \ra 0$, this relates to the $\infty$-obstacle problem. The constant bounding the limsup in \eqref{eq_theta} is sharp, as readily seen on flat Euclidean space by noting that 
	\[
	u(x) = \sqrt[1-\theta]{\lambda\frac{(1-\theta)^2}{2(1+\theta)}} |x|^{\frac{2}{1-\theta}}
	\]
solves $\Delta_\infty^N u = \lambda u^\theta$.		
}
\end{remark}

\begin{remark}[\textbf{On conditions $\bf 7),8)$}]
\emph{The equivalence between $1)$ and $7),8)$ were first pointed out in \cite[Thms. 2.28 and 2.29]{pigolasetti_ensaio} in a Riemannian setting: it is inspired by the characterization of parabolic Riemannian manifolds by means of the vanishing of the $2$-capacity $\capac_2(K)$ of some/every compact set $K$ (cf. \cite{grigoryan}), and to equivalent ones for the $p$-Laplacian, $p\in (1,\infty)$ in terms of the $p$-capacity
	\[
	\capac_p(K) : = \left\{ \int_M |\nabla u|^p \ : \ u \in \lip_c(M), \ u \ge 1 \ \text{ on } \, K \right\}.
	\]
Observe that, to detect the \emph{forward} completeness, we had to switch signs and define our class $\mathscr{L}(K,M)$ by requiring $u \le -1$ on $K$. 	 
}
\end{remark}

\begin{remark}[\textbf{Normalized vs unnormalized $\infty$-Laplacian}]
\emph{The equivalence between items $1), \ldots, 5)$ could be rephrased for the unnormalized $\infty$-Laplacian with minor changes, replacing $\Delta_\infty^Nu \ge g(u)$ with the inequality 
	\[
	\Delta_\infty u \ge g(u)|\nabla u|^2,
	\]
and $5)$ with the following statement:} 
\begin{itemize}
\item[5')] for some/every $\theta \in (0,3)$ and $\lambda > 0$, it holds	
	\[
	\left\{ \begin{array}{l}
	\Delta_\infty u \ge \lambda u_+^\theta \qquad \text{on } \,  M, \\[0.4cm]
	\disp \limsup_{\varrho^+(x) \ra +\infty} \frac{u_+(x)}{\varrho^+(x)^{\frac{4}{3-\theta}}} < \sqrt[3-\theta]{\lambda\frac{(3-\theta)^4}{64(1+\theta)}}  
	\end{array}\right. \qquad \Longrightarrow \qquad \text{$u$ is a (nonpositive) constant}.
	\]
\end{itemize}
\end{remark}

The fact that the forward completeness of $(M,F)$ implies any of $2), \ldots, 4)$ is not difficult to prove, and might be well-known among specialists, although we found no precise reference; on the other hand, $1) \Rightarrow 5)$ is more subtle, due to the possibility that the limsup in \eqref{eq_theta} be positive, and  inspired by \cite{araujo_leitao_teixeira}. We briefly comment on implications $8) \Rightarrow 1)$ and $3) \Rightarrow 1)$, that are the technical core of the present work.\par
The proof of $8) \Rightarrow 1)$ exploits results in \cite{marivaltorta,maripessoa}, namely it uses the Ahlfors-Khas'minskii duality (AK-duality, for short). Roughly speaking, for a large class of fully nonlinear inequalities 
	\begin{equation}\label{eq_gensub}
	\mathscr{F}(x,u, \di u, \hess u) \ge 0, 
	\end{equation}
the AK-duality establishes the equivalence between a maximum principle at infinity for solutions of \eqref{eq_gensub}, in the form given by \eqref{teo_WMP} (called there the Ahlfors property), and the existence of solutions of the \emph{dual} inequality
	\[
	\widetilde{\mathscr{F}}(x,u, \di u, \hess u) \ge 0, \qquad \text{with} \ \ \widetilde{\mathscr{F}}(x,r,p, A) = - \mathscr{F}(x,-r,-p,-A),
	\]
that decay to $-\infty$ as slow as we wish\footnote{We say that $u$ decays to $-\infty$ if upper level sets of $u$ have compact closure in $M$.} (named Khas'minskii potentials). The eikonal equation
	\[
	G(u) - F(\nabla u) = 0
	\] 	
falls into the class of PDEs for which the AK-duality holds, thus we can construct a Khas'minskii potential $w$ that is a subsolution of the dual equation $\widetilde{F}(\widetilde{\nabla} w) - \widetilde{G}(w) = 0$, with $\widetilde{F}$ the dual Finsler structure, $\widetilde{\nabla}$ the gradient induced by $\widetilde{F}$ and $\widetilde{G}(t) : = G(-t)$. The existence of $w$ easily implies the forward completeness of $M$. The construction of $w$ proceeds, as in \cite{maripessoa,maripessoa_2}, by stacking solutions of obstacle problems, and has independent interest.\par  
Implication $3) \Rightarrow 1)$ is shown by means of a sequence $\{u_j\}$ of solutions of $\Delta_\infty^N u_j = g(u_j)$ defined on an increasing family of relatively compact sets $\Omega_j$, locally converging to a limit solution $u_\infty$ on $M \backslash K$, with $K$ a small compact set. The main issue is to prove a uniform, global Lipschitz bound for $\{u_j\}$ without knowing that $M$ be forward complete. In fact, one cannot use the classical local Lipschitz bound for bounded solutions of $\Delta^N_\infty u \ge 0$ as in \cite[Lemma 2.5]{crandall_evans_gariepy}, since the latter is uniform for $u \in L^\infty(M)$ only if balls of a fixed radius centered at any point of $M$ are relatively compact, that force $M$ to be forward complete. In \cite{maripessoa}, for $g \equiv 0$, the authors reach the goal by exploiting the absolutely minimizing property of the $\infty$-harmonic functions $u_j$, a characterization that currently seems unavailable\footnote{In this respect, note that \eqref{eq_cong} is not included in the class of PDEs considered in \cite{barron_jensen_wang}, where the authors compute the Euler-Lagrange equations of absolute minimizers for 
	\[
	\mathscr{I}(u,\Omega) = \mathrm{ess} \sup_{x \in \Omega} f(x,u(x),\di u(x))
	\]
In our case (say, even in a Riemannian setting), the PDE $\Delta_\infty u = g(u)|\nabla u|^2$ for the unnormalized $\infty$-Laplacian would be, formally, the Euler-Lagrange equation for the choice
	\[
	f(x,s,p) = |p|^2 - 2 \int_0^s g(t) \di t,
	\] 		
a function that does not satisfy all of the assumptions in Theorem 3.5 of \cite{barron_jensen_wang}.}  for solutions of $\Delta_\infty^Nu = g(u)$. We overcome the problem by showing a Lipschitz bound directly via comparison with radial solutions $g$ (hereafter called $g$-cones), extending an elegant argument in \cite[Prop. 2.1]{aronsson_crandall_juutinen}. We prove the following result that, to the best of our knowledge, seems to be new.
	
\begin{theorem}\label{teo_fundamental_intro}
Let $\Omega \Subset (M,F)$, and let $u \in C(\overline{\Omega})$ satisfy
	\begin{equation}\label{eq_cong}
	\Delta_\infty^N u = g(u) \quad \text{on } \, \Omega,
	\end{equation}
where $g$ is continuous and non-negative on $u(\overline{\Omega})$. If $u$ is Lipschitz on $\partial \Omega$, then $u \in \lip(\overline{\Omega})$ and 
	\begin{equation*}
	\lip(u, \Omega) \le \sqrt{ \lip(u, \partial \Omega)^2 + 2\int^{\sup_\Omega u}_{\inf_\Omega u} g(s)\di s }.
	\end{equation*}
	\end{theorem}	
\noindent In the particular case $g \equiv 0$, this reduces to the AMLE condition $\lip(u, \Omega) = \lip(u, \partial \Omega)$. 
\vspace{0.2cm}


The paper is organized as follows: in Section \ref{sec_prelim} we recall definitions and main properties of Finsler manifolds. In Sections \ref{sec_viscosity} and \ref{sec_compacones}, we define viscosity solutions of $\infty$-Laplace equations, state their main comparison results with forward and backward $g$-cones, and prove Theorem \ref{teo_fundamental_intro}. Eventually, in Section \ref{sec_proof} we prove Theorem \ref{teo_main}. Appendices I and II contain some ancillary results adapted to the Finsler setting. 
	
\vspace{0.5cm}

\noindent \textbf{Acknowledgements.} The authors would like to express their gratitude to Andrea Mennucci, for valuable suggestions. The first and third authors are partially supported by CNPq-Brazil, as well as by FAPESQ-PB (2019/0014) and PROMISS\~OES-UFPI (010/2018), respectively. They also would like to thank the worm hospitality of the Abdus Salam International Centre for Theoretical Physics (ICTP), and of the Mathematisches Forschungsinstitut Oberwolfach (MFO), where part of this work was conducted.

%

\section{Basics on Finsler manifolds}\label{sec_prelim}

Let $M$ be an $m$-dimensional smooth manifold. As usual we denote by $TM \doteq \cup_{x\in M} T_x M$ the tangent bundle of $M$, where $T_x M$ means the tangent space at $x\in M$. Each element of $TM$ has the form $(x,p)$, where $x\in M$ and $p = p^i\frac{\partial}{\partial x^i} \in T_xM$. A Finsler structure on $M$ (cf. \cite{bao_chern_shen}) is a function $F : TM \rightarrow [0,\infty)$ satisfying the following properties:
\begin{enumerate}
\item[i)] \emph{Regularity:} $F$ is smooth on $TM\backslash 0$, with $0$ the zero section.
\item[ii)] \emph{Positive homogeneity:} $F(x,\lambda p) = \lambda F(x,p)$ for all $\lambda>0$.
\item[iii)] \emph{Strong convexity:} The fundamental tensor
	\[
	g_{ij}(x,p) := \frac{1}{2} \frac{\partial^2 F^2(x,p)}{\partial p^i \partial p^j}
	\]
is positive definite at every $(x,p) \in TM\backslash 0$.
\end{enumerate} 
Note that the expression $g_{ij}(x,p)p^i p^j$ is invariant by a change of coordinates.
%
We call a Finsler manifold the pair $(M,F)$, where $M$ is a smooth manifold and $F$ is a Finsler structure on $M$. Riemannian manifolds $(M,g)$ are a particular subclass of Finsler manifolds, obtained by choosing
	\[
	F(x,p) : = \sqrt{ g_{ij}(x)p^i p^j}.
	\] 
The induced Finsler structure $F^* : T^*M \ra [0, \infty)$ on the cotangent bundle is defined by 	
	\[
	F(x,\xi) \doteq \sup_{p \in T_xM\backslash 0}\frac{\xi(p)}{F(x,p)} = \sup_{F(x,p)=1}\xi(p),
	\]
and gives rise to a family of Minkowski norms $F^* = \{F^*_x\}_{x \in M}$ with corresponding fundamental tensor  	\[
	g^{* kl}(\xi) = \frac{1}{2}\frac{\partial^2 F^{* 2}(\xi)}{\partial \xi_k \partial \xi_l}.
	\]
Hereafter, we write $F(p), F^*(\xi)$ for notational convenience, suppressing the dependence on $x$. 
	%
%
We will use the Chern connection of $(M,F)$, defined on the vector bundle $\pi^{*}TM$, where $\pi : TM\backslash 0 \rightarrow M$ is the natural projection. Its connection forms are torsion free, that is,
$$ dx^j \wedge \omega^i_j = 0,$$
which means that $dp^k$ are absent in the definition of $\omega^i_j$, namely,
$$ \omega^i_j = \Gamma^i_{jk} dx^k, \quad \text{and} \quad \Gamma^i_{jk} = \Gamma^i_{kj}.$$

Let $\Omega \subset M$ be open and consider a coordinate system $(x^i,\frac{\partial}{\partial x^i})$ on $T\Omega$. Given a non-vanishing vector field $v = v_i \frac{\partial}{\partial x^i}$ on $\Omega$, we introduce a Riemannian metric $g_v$ and a linear connection $\nabla^v$ on $T \Omega$ by setting, for $p = p^i \frac{\partial}{\partial x^i}$ and $q = q^i \frac{\partial}{\partial x^i}$ in $T_x \Omega$, 
	\[
	g_v(p,q) \doteq p^i q^j g_{ij}(x,v), \quad \text{and} \quad \nabla^v_{\frac{\partial}{\partial x^i}} \frac{\partial}{\partial x^j} \doteq \Gamma^k_{ij}(x,v)\frac{\partial}{\partial x^k}.
	\]
We define the Legendre transformation $\ell : T M \rightarrow T^* M$ by
\[
\ell(p) = \left\lbrace
\begin{array}{ll}
g_{p}(p,\cdot), & p\not=0,\\
0, & p=0.
\end{array}\right.
\]
Remarkably, $\ell : TM\backslash 0 \rightarrow TM^*\backslash 0$ is a smooth diffeomorphism and 
	\[ 
	F^*(\ell(p)) = F(p), \quad \text{for all} \ \ p \in TM.
	\] 
Consequently, $g^{*ij}(\ell(p))$ coincides with the inverse of $g_{ij}(p)$ (see \cite{bao_chern_shen}, \cite{shen}), and the map $\ell^{-1}: T^*M \ra TM$ does exist. Given a smooth function $f: M \rightarrow \R$, we therefore define the gradient of $f$ as 
	\[
	\nabla f = \ell^{-1}(\di f).
	\]
In particular, note that
	\[
	\begin{array}{l}
	\di f(p) \le F^*(\di f)F(p) = F(\nabla f)F(p) \quad \forall \, f \in C^1(M), \ p \in TM, \\[0.2cm]
	\di f(p) = g_{\nabla f}(\nabla f,p) \ \ \text{on $\, \mathcal{R}_f= \big\{ x : \di_x f \neq 0\big\}$, for all} \ \ p \in TM.
	\end{array}
	\] 
Following \cite{wu_xin}, given a smooth function $f$ we define its Hessian $\Hess f$ on $\mathcal{R}_f$ by 
	\[
	\Hess f (V,W) \doteq V W(f) - \nabla^{\nabla f}_{V} W(f), \quad \text{for all} \ \ V, W \in T\mathcal{R}_f.
	\]
It is easy to see that $\Hess f$ is symmetric and can be rewritten as 
$$ \Hess f (V,W) = g_{\nabla f}\left(\nabla^{\nabla f}_{V}\nabla f, W\right).$$ 

An alternative construction is proposed in \cite{shen}, where the Hessian of $f$ is defined as the map
	\[
	D^2 f : TM \ra \R, \qquad D^2 f(p) \doteq \frac{d^2}{ds^2}\left(f\circ \gamma\right)_{|_{s=0}},
	\]
with $\gamma : (-\varepsilon,\varepsilon) \rightarrow M$ the geodesic satisfying $\gamma'(0) = p$. 
In \cite{wu_xin}, the authors point out that
	\[
	D^2 f(V) \equiv \Hess f (V,V), \quad \text{for all} \ \ V \in T\mathcal{R}_f.
	\]
	

\subsection{Forward and backward completeness}

For $x_0, x_1 \in M$, denote by $\Gamma(x_0,x_1)$ the collection of all piecewise smooth curves $\gamma : [a,b] \rightarrow (M,F)$ with $\gamma(a) = x_0$ and $\gamma(b) = x_1$. The distance $\di : M\times M \rightarrow [0,\infty)$ is defined by 
	\[
	\di(x_0,x_1) \doteq \inf_{\Gamma(x_0,x_1)}L(\gamma), \qquad \text{with } \, L(\gamma) : = \int_a^b F(\gamma'(t))\di t
	\]
the length of $\gamma$. Despite $\di$ is not a metric, the space $(M,\di)$ satisfies the first two axioms of a metric space:
\begin{enumerate}
\item $\di(x_0,x_1) \geq 0$, with equality holding iff $x_0 = x_1$.
\item $\di(x_0,x_2) \leq \di(x_0,x_1) + \di(x_1,x_2)$.
\end{enumerate}
The symmetry $\di(x_0,x_1) = \di(x_1,x_0)$ is satisfied whenever the Finsler structure $F$ is absolutely homogeneous, that is $F(\lambda p) = \lambda F(p)$ for every $\lambda \in \R$. In this case, $(M,\di)$ is a genuine metric space.

For $\bar x \in M$ fixed, and $r>0$, we define on $T_{\bar x}M$ the tangent balls and spheres of radius $r$
$$ 
B_{\bar x}(r) : = \big\{p \in T_{\bar x} M : F(\bar x,p) < r \big\}, \quad S_{\bar x}(r) : = \big\{p \in T_{\bar x} M : F(\bar x,p) = r\big\},
$$
and the corresponding \emph{forward metric balls and spheres} 
$$ 
\mathcal{B}^{+}_{\bar x}(r) : = \big\{x\in M : \di(\bar x,x) <r\big\}, \quad \mathcal{S}^{+}_{\bar x}(r) : = \disp \big\{x\in M : \di(\bar x,x) = r \big\}. 
$$ 
The associated backward balls and spheres
	\[
	\mathcal{B}^{-}_{\bar x}(r) : = \big\{x\in M : \di(x,\bar x) <r\big\}, \quad \mathcal{S}^{-}_{\bar x}(r) : = \disp \big\{x\in M : \di(x,\bar x) = r \big\} 
	\]
coincide with the forward balls of the dual Finsler structure $\widetilde{F}$. As proved in Section 6.2 C of \cite{bao_chern_shen}, the topology of the underlying manifold and that generated by the forward balls coincide. Hence we can state that a sequence $x_i \ra x$ in $M$ if, given any open set $O \ni x$, there is a positive integer $N$ (depending on $O$) such that $x_i \in O$ whenever $i\geq N$. According to Lemma 6.2.1 in \cite{bao_chern_shen}, for a fixed point $x_0 \in M$ there exist an open neighbourhood $U$ and a constant $\alpha >1$, depending on $x_0$ and $U$, such that
\begin{equation}\label{reversible_metric_ineq}
\frac{1}{\alpha}\di(x_2,x_1) \leq \di(x_1,x_2) \leq \alpha \di(x_2,x_1) \qquad \forall \, x_1, x_2 \in U.
\end{equation}
Therefore, the statements 
$$ 
x_i \rightarrow x, \qquad \di(x,x_i) \rightarrow 0, \qquad \di(x_i,x) \rightarrow 0$$
are equivalent. However, this is not the case in general for Cauchy sequences.

\begin{definition}
A sequence $\{x_i\}$ in $M$ is called a \emph{forward (resp., backward) Cauchy sequence} if, for all $\varepsilon>0$, there exists a positive integer $j_\eps$ (depending on $\varepsilon$) such that
$$ 
j_\eps \leq i < j \Longrightarrow \di(x_i,x_j) < \varepsilon \qquad [\text{resp.}, \di(x_j,x_i) < \varepsilon].
$$
\end{definition}

\begin{definition}
A Finsler manifold $(M,F)$ is said to be \emph{forward complete} if every forward Cauchy sequence converges in $M$. It is said to be \emph{backward complete} if every backward Cauchy sequence converges.
\end{definition}

A geodesic $\gamma$ from $\bar x$ to $x$ is a curve that is stationary for $L$. It can (and will henceforth) be reparametrized via an affine map to have constant velocity $F(\gamma') \equiv 1$. The exponential map $\exp_{\bar x}$ associates to $v \in T_{\bar x}M$ the value $\gamma_v(1)$ of the unique forward geodesic $\gamma_v$ issuing from $\bar x$ with constant velocity $F(v)$. The following result summarizes the minimizing properties of short geodesics that we need.
%

\begin{theorem}
Let $(M,F)$ be a Finsler manifold. Then, for a given compact set $K$, there exists $\eps>0$ such that
\begin{enumerate}
\item[1)] \cite[pp. 126-127]{bao_chern_shen} The map
	\[\exp \ \ : \ \big\{ v \in TK : F(v) < \eps\big\} \ra M, \qquad \exp(x,v) = \exp_x(v)
	\] 
is a $C^1$-diffeomorphism onto its image, and $C^\infty$ outside of the zero section. 
\end{enumerate}
Fix a point $\bar x$ and suppose that, for some $r, \eps>0$, $\exp_{\bar x}$ is a $C^1$-diffeomorphism from the tangent ball $B_{\bar x}(r+\varepsilon)$ onto its image (we call these balls \emph{regular}). Then:
\begin{enumerate}
\item[2)] \cite[Thm. 6.3.1]{bao_chern_shen} Each radial geodesic $\exp_{\bar x}(tv)$, $0\leq t\leq r$, $F(\bar x,v) = 1$ is the unique curve that minimizes distance among all piecewise $C^\infty$ curves in $M$ with the same endpoits.
\end{enumerate}
\end{theorem}

%

The corresponding behaviour of the distance function from (or towards) a fixed origin $\bar x \in M$ on small balls has been described in \cite{shen}, Lemma 3.2.4, and in \cite{wu_xin}, equation $(4.1)$. Summarizing, we have

\begin{proposition}\cite{shen, wu_xin}\label{prop_nabla_hess}
Let $(M,F)$ be a Finsler manifold, let $r>0$ be such that $\mathcal{B}_{\bar x}^+(r)$ and $\mathcal{B}_{\bar x}^-(r)$ are regular geodesic balls. Then, the functions
	\[
	\varrho^+(y) = \di(\bar x, y), \qquad \varrho^-(y) = - \di(y, \bar x)
	\]
are smooth on, respectively, $\mathcal{B}_{\bar x}^+(r)\backslash \{\bar x\}$ and $\mathcal{B}_{\bar x}^-(r)\backslash \{\bar x\}$, and there they satisfy
	\[
	F(\nabla \varrho^{\pm}) = 1, \qquad \Hess \varrho^\pm (\nabla \varrho^\pm, \nabla \varrho^\pm) = 0. 	
	\]
\end{proposition}

Indeed, the identity $F(\nabla \varrho^\pm) = 1$ is proved in \cite[Lem 3.2.4]{shen}, while for the Hessian identity we observe the following: if $\gamma : [0, \di(y, \bar x)] \ra \mathcal{B}_{\bar x}^-(r)$ is a geodesic from $y$ to $\bar x$ with initial velocity $\nabla \varrho^-(y)$, then $\varrho^-(\gamma(t)) = - \di(\gamma(t),\bar x) = t - \di(y,\bar x)$ and thus
	\[
	\Hess \varrho^-(\nabla \varrho^-, \nabla \varrho^-) = \frac{\di^2}{\di t^2} \varrho^-(\gamma(t)) = 0.
	\]
	
Regarding the behaviour of long minimizing geodesics, we have the following Hopf-Rinow type theorem due to Cohn-Vossen \cite{cohnvossen} (cf. also \cite{mennucci, mennucci_2} for more general statements, also considering Finsler metrics constructed from Hamilton-Jacobi equations).

%

\begin{theorem}[\cite{cohnvossen}, see Theorem 6.6.1 in \cite{bao_chern_shen}]
Let $(M,F)$ be a connected Finsler manifold. The following properties are equivalent:
\begin{enumerate}
\item[1.] $(M,F)$ is forward complete.
\item[2.] $(M,F)$ is forward geodesically complete, that is, every geodesic $\gamma(t)$, $a\leq t\leq b$, parametrized to have constant speed, can be extended to a geodesic defined on $a\leq t< \infty$.
\item[3.] For some/every $x\in M$, $\exp_x$ is defined on all of $T_xM$.
\item[4.] Every closed and forward bounded subset $K \subset M$ (in the sense that $K$ is contained into some forward ball) is compact.
\end{enumerate}
Furthermore, if any of the above holds, then every pair of points in $M$ can be joined by a minimizing geodesic. 
\end{theorem}

%


\section{Viscosity solutions}\label{sec_viscosity}

Hereafter, given a test function $\phi$ regular enough, with $\phi \prec_x u$ (resp., $\phi \succ_x u$) we mean that $\phi$ is defined in a neighbourhood of $x$, $\phi \le u$ (resp. $\phi \ge u$) and $\phi(x) = u(x)$. We start by recalling the definition of subsolutions for the eikonal equations. 

\begin{definition}
Given $\Omega \subset M$ open and $G \in C(\Omega \times \R)$, we say that 
\begin{itemize}
\item[1.] $u \in \USC(\Omega)$ is a viscosity subsolution of 
	\[
	F(\nabla u) - G(x,u) = 0 \quad \text{on } \, \Omega
	\]
if, for every $x \in \Omega$ and test function $\phi \succ_x u$ of class $C^1$ it holds $F(\nabla \phi) - G(x,\phi) \le 0$ at $x$.
\item[2.] $u \in \USC(\Omega)$ is a viscosity subsolution of 
	\[
	G(x,u)- F(\nabla u) = 0 \quad \text{on } \, \Omega
	\]
if, for every $x \in \Omega$ and test function $\phi \succ_x u$ of class $C^1$ it holds $G(x,\phi)- F(\nabla \phi) \le 0$ at $x$.
\end{itemize}
\end{definition}

\noindent Next, for $\phi \in C^2(\Omega)$ we define
\[
\Delta^{N,+}_{\infty} \phi(x) = \left\lbrace
\begin{array}{ll}
\Hess \phi\left(\frac{\nabla \phi}{F(\nabla \phi)},\frac{\nabla \phi}{F(\nabla \phi)}\right), & \text{if} \ \ \di_x \phi \not= 0,\\[0.2cm]
\max\left\lbrace D^2 \phi(p,p) : F(p) = 1\right\rbrace, &  \text{if} \ \ \di_x \phi = 0.
\end{array}\right.
\]
and
	\[
	\Delta^{N,-}_{\infty} \phi(x) = \left\lbrace
	\begin{array}{ll}
	\Hess \phi\left(\frac{\nabla \phi}{F(\nabla \phi)},\frac{\nabla \phi}{F(\nabla \phi)}\right), & \text{if} \ \ \di_x \phi \neq 0,\\[0.2cm]
	\min\left\lbrace D^2 \phi(p,p) : F(p) = 1\right\rbrace, &  \text{if} \ \ \di_x \phi = 0.
	\end{array}\right.
	\]


\begin{definition}
Let $\Omega \subset M$ be open, and let $f : \mathbb{R}\times T^*\Omega \rightarrow \mathbb{R}$ be a continuous function (the dependence of $f$ on $x \in \Omega$ is implicit when writing $T^*\Omega$). 
\begin{enumerate}
\item[1.] A function $u \in \USC(\Omega)$ is said to solve $\Delta^N_\infty u \ge f(u,\di u)$ 
	\begin{itemize}
		\item[$\bullet$] in the \emph{viscosity sense} if, for every $x \in \Omega$ and every test function $\phi \succ_x u$ of class $C^2$,
	\[
	\Delta^{N,+}_{\infty} \phi \geq f(\phi(x), \di \phi(x));
	\]
		\item[$\bullet$] in the \emph{barrier sense} if, for every $x \in \Omega$, there exists $u_\eps \in C^2$ with $u_\eps \prec_x u$ and 
	\[
	\Delta_\infty^{N,+} u_\eps \ge f(u_\eps(x), \di u_\eps(x)) - \eps.
	\]	
	\end{itemize}
	In these cases, we also say that $u$ is a subsolution (in the viscosity/barrier sense). 	
\item[2.] A function $u \in \LSC(\Omega)$ is said to solve $\Delta^N_\infty u \le f(u,\di u)$ 
	\begin{itemize}
		\item[$\bullet$] in the \emph{viscosity sense} if, for every $x \in \Omega$ and every test function $\phi \prec_x u$ of class $C^2$, 
	\[
	\Delta^{N,-}_{\infty} \phi \leq f(\phi(x),\di \phi(x));
	\]
		\item[$\bullet$] in the \emph{barrier sense} if, for every $x \in \Omega$, there exists $u_\eps \in C^2$ with $u_\eps \succ_x u$ and 
	\[
	\Delta_\infty^{N,-} u_\eps \le f(u_\eps(x), \di u_\eps(x)) + \eps.
	\]	
	\end{itemize}
	In these cases, we also say that $u$ is a supersolution (in the viscosity/barrier sense).	
\item[3.] A function $u \in C(\Omega)$ is said to solve 
\begin{equation}\label{main_equation}
\Delta_{\infty}^N u = f(u,\di u) \quad \text{on } \, \Omega
\end{equation}
(in the viscosity/barrier sense) if it is both a subsolution and a supersolution. 
\end{enumerate}
\end{definition}

\begin{remark}
\emph{If $u$ is a subsolution (resp. a supersolution) in the barrier sense, and $f$ is continuous, then $u$ is also a subsolution (supersolution) in the viscosity sense. However, the converse is not necessarily true. 
}
\end{remark}


In the following proposition we state useful properties satisfied by $\infty$-Laplacian subsolutions, that in our needed generality (the operator is discontinuous) can be found in \cite[Thm. 2.6]{HL_dir} and \cite[Prop. 3.7]{mebrate_mohammed}.

	
\begin{proposition}\label{elementary_prop_sub}
Let $\Omega \subset M$ be a bounded subset and $f \in C(\R\times T^*\Omega)$.
\begin{enumerate}
\item[i)] If $u, v \in \USC(\Omega)$ are subsolutions of \eqref{main_equation}, then $\max\{u,v\}$ is also a subsolution of \eqref{main_equation}.
\item[ii)] {\rm (Stability)} If $\{u_k\} \subset \USC(\Omega)$ is a sequence of \emph{viscosity} subsolutions of \eqref{main_equation}, and $u_k \rightarrow u$ converges locally uniformly in $\Omega$, then $u$ is also a viscosity subsolution of \eqref{main_equation}.
\end{enumerate}
\end{proposition}

%

\subsection{Calabi's trick} 

We begin with a chain rule for the $\infty$-Laplacian. Let $\eta \in C^2(\mathbb{R})$ and $\phi \in C^2(\Omega)$, where $\Omega \subset M$ is an open set. Since the function $w = \eta \circ \phi$ solves
\begin{equation}\label{chain_rule}
\Delta^{N, \pm}_{\infty} w = \eta''(\phi)F^2(\nabla \phi) + \eta'(\phi) \Delta^{N,\pm}_{\infty} \phi \qquad \text{on } \, \Omega^* = \Big\{ x \in \Omega : \eta'(\phi(x)) >0\Big\},
\end{equation}
a direct check shows the following

\begin{proposition}\label{prop_composition}
Let $u \in \USC(\Omega)$ (resp., $\LSC(\Omega))$ be a subsolution (resp., a supersolution) of \eqref{main_equation}, and let $\eta \in C^2(\mathbb{R})$. On the set $\Omega^* = \{ x \in \Omega : \eta'(u) >0\}$, the function $w = \eta \circ u$ is a viscosity subsolution (resp., supersolution) of 
	\[
	\Delta_{\infty}^N w = \eta''(u) F^2(\nabla u) + \eta'(u)f(u,\di u). 
	\]
\end{proposition}


The following Lemma is a form of the classical Calabi's trick \cite{calabi} adapted to the Finsler setting. By slightly modifying the original argument, we are able to avoid the assumption that the underlying manifold be forward complete, a fact that will be important in what follows.

\begin{lemma}[Calabi's trick] \label{lem_calabi}
Let $(M,F)$ be a Finsler manifold, fix $\bar x \in M$ and define
	\[
	\varrho^+(y) = \di(\bar x,y), \qquad \varrho^-(y) = -\di(y, \bar x) \qquad \forall \, y \in M.
	\]
Let $x \in M\backslash \{\bar x\}$. Then, for every $\eps > 0$ small enough there exist functions $\varrho_\eps^+, \varrho_\eps^-$  satisfying the following properties:
	\begin{equation}\label{proprie_calabi}
	\left\{ \begin{array}{l}
	\varrho_\eps^+, \varrho_\eps^- \ \ \text{ are smooth in a neighbourhood $U_\eps$ of $x$,} \\[0.2cm]
	\varrho_\eps^+ \succ_x \varrho^+, \quad \varrho_\eps^- \prec_x \varrho^- \\[0.3cm]
	F(\nabla \varrho_\eps^\pm) = 1, \quad \Hess \varrho_\eps^\pm \left(\nabla \varrho_\eps^\pm,\nabla \varrho_\eps^\pm\right) = 0 \ \ \text{on} \ \ U_\eps.
	\end{array}
	\right.
	\end{equation}
In particular, for every $\eta \in C^2(\R)$, the functions $w_\eps^\pm = \eta( \varrho_\eps^\pm)$ satisfy 
	\begin{equation}\label{eq_whenrhosmooth1}
	F(\nabla w_\eps^\pm) = \eta'(\varrho_\eps^\pm), \quad \Delta_\infty^{N,\pm} w_\eps^\pm = \eta''(\varrho_\eps^\pm) \qquad \text{on} \quad U^* \doteq \{ x \in U_\eps : \eta'(\varrho_\eps^\pm) >0\}.
	\end{equation}
\end{lemma}

\begin{proof}
We first prove the statement for $\varrho^+$. Fix a small $\eps>0$ in such a way that 
\begin{itemize}
\item[(i)] the backward geodesic ball $\mathcal{B}_x^-(2\eps)$ is relatively compact. 
\item[(ii)] for every $y \in \mathcal{B}_x^-(2\eps)$, $\exp_y : B_{y}^+(2\eps)\subset T_yM \ra \mathcal{B}_y^+(2\eps)$ is a diffeomorphism.
\end{itemize}
Choose $x_\eps \in \mathcal{S}_x^-(\eps)$ to be the minimum point of $\varrho^+$ restricted to $\mathcal{S}_x^-(\eps)$, and define
	\[
	\varrho_\eps^+(y) \doteq \di(\bar x,x_\eps) + \di(x_\eps, y) \qquad \forall \, y \in M.
	\]
By the triangle inequality, $\varrho_\eps^+ \ge \varrho^+$ on $M$. We claim that equality holds at $y=x$. Indeed, assume by contradiction that $\varrho_\eps^+(x) = \varrho_+(x) + c_\eps$ for some $c_\eps>0$. Let $\{\gamma_j\}$ be a sequence of unit speed curves from $\bar x$ to $x$ with $L(\gamma_j) \le \varrho^+(x) + j^{-1}$ and, for every $j$, define 
	\[
	t_j = \inf \Big\{ t \in [0, L(\gamma_j)] \ : \ \gamma_j\big((t_j,L(\gamma_j)]\big) \subset \mathcal{B}^-_x(\eps) \Big\}. 
	\]
	Note that $x_j= \gamma(t_j) \in \mathcal{S}^-_x(\eps)$.	Then, 
	\[
	\begin{array}{lcl}
	\di(\bar x, x) + \frac{1}{j} & \ge & \disp L(\gamma_j) = L\Big((\gamma_j)_{[0,t_j]}\Big) + L\Big((\gamma_j)_{[t_j,L(\gamma_j)]}\Big) \\[0.3cm]
	& \ge & \di(\bar x, x_\eps) + \di(x_j, x) = \di(\bar x, x_\eps) + \di(x_\eps, x) > \di(\bar x,x) + c_\eps,
	\end{array}
	\]
a contradiction if $j$ is chosen to be large enough.

Having shown that $\varrho_\eps^+$ touches $\varrho^+$ from above at $x$, by (ii) we deduce that $\varrho_\eps^+$ is smooth on $\mathcal{U}_\eps \doteq \mathcal{B}_{x_\eps}^+(2\eps)\backslash \{x_\eps\}$, that is a neighbourhood of $x$. Moreover, by Proposition \ref{prop_nabla_hess}
	\begin{equation*}\label{ipo_rhoeps}
	F(\nabla \varrho_\eps^+) = 1, \qquad \Hess \varrho_\eps^+\left(\nabla \varrho_\eps^+,\nabla \varrho_\eps^+\right) = 0 \qquad \text{on } \, \mathcal{U}_\eps,
	\end{equation*}
as required. The argument is analogous for the signed distance $\varrho^-$: we choose $\eps$ small enough to match
\begin{itemize}
\item[(i)] the forward geodesic ball $\mathcal{B}_x^+(2\eps)$ is relatively compact. 
\item[(ii)] for every $y \in \mathcal{B}_x^+(2\eps)$, $\exp_y : B_{y}^-(2\eps)\subset T_yM \ra \mathcal{B}_y^-(2\eps)$ is a diffeomorphism.
\end{itemize}
Choose then $x_\eps \in \mathcal{S}_x^+(\eps)$ minimizing $-\varrho^- = \di(\cdot, \bar x)$ on $\mathcal{S}_x^+(\eps)$ and define $\varrho_\eps^-$ according to the identity
	\[
	-\varrho_\eps^-(y) : = \di(y,x_\eps) + \di(x_\eps,\bar x) \ge -\varrho^-(y) \qquad \forall \, y \in M. 
	\]
With the same argument as above, we can show that $\varrho_\eps^- \prec_x \varrho^-$, and the third condition in \eqref{proprie_calabi} follows from Proposition \eqref{prop_nabla_hess} as well. To conclude, on $U^*$ it holds $F(\nabla w_\eps^\pm) = \eta'(\varrho_\eps^\pm) F(\nabla \varrho_\eps^\pm) = \eta'(\varrho_\eps^\pm)$, while from equation \eqref{chain_rule},
\begin{eqnarray*}
\Delta_{\infty}^{N,\pm}
 w^{\pm}_\eps &=& \eta''(\varrho^\pm_\eps)F^2(\nabla \varrho^\pm_\eps) + \eta'(\varrho^\pm_\eps)\Delta^{N,\pm}_\infty \varrho^\pm_\eps = \eta''(\varrho^\pm_\eps).
\end{eqnarray*}
\end{proof}

\begin{corollary}\label{cor_radial_inf}
Let $(M,F)$ be a Finsler manifold, and $\eta \in C^2(\R)$. Fix $\bar x \in M$ and consider the signed distance functions 
	\[
	\varrho^+(\cdot) = \di(\bar x, \cdot), \qquad \varrho^-(\cdot) = -\di(\cdot,\bar x). 
	\]
Then, $v := \eta(\varrho^{+})$ is a viscosity supersolution of $F(\nabla v) - \eta'(\varrho^+) = 0$ on $\big\{\eta'(\varrho^+)>0 \big\} \backslash \{\bar x\}$ (that is, $F(\nabla \phi) - \eta'(\varrho^+) \ge 0$ holds at $x$ whenever $\phi \prec_x v$), and there it satisfies 
	\[
	\Delta^N_\infty v \le \eta''(\varrho^+) 
	\]
in the barrier sense. Similarly, the function $u: = \eta(\varrho^-)$ is a viscosity subsolution of $F(\nabla u) - \eta'(\varrho^+) = 0$, and it satisfies
	\[
	\Delta^N_\infty u \ge \eta''(\varrho^-)
	\]
in the barrier sense on $\big\{ \eta'(\varrho^-)>0 \big\} \backslash \{\bar x\}$. 
\end{corollary}

\begin{proof}
We will just prove it for $\varrho^+$. Let $\varrho_\eps^+$ be defined as in Lemma \ref{lem_calabi} and smooth in a neighbourhood $U_\eps$. Up to reducing $\eps$, we can further assume that $\eta'(t)>0$ for every $t \in [\varrho^+(y),\varrho^+_\eps(y)]$ and $y \in U_\eps$. Therefore, $v_\eps \doteq  \eta(\varrho_\eps^+) \succ_x v$ and 
	\[
	F(\nabla v_\eps) = \eta'(\varrho_\eps^+) = \eta'(\varrho^+), \qquad \Delta_\infty^{N,-} v_\eps = \eta''(\varrho^+_\eps) = \eta''(\varrho^+) \qquad \text{at } \, x.
	\]
If $\phi \prec_x v$, then $\nabla \phi(x) = \nabla v_\eps(x)$ and thus $F(\nabla \phi) - \eta'(\varrho^+) = 0$ at $x$.
		
\end{proof}

\section{Comparison with $g$-cones and Lipschitz regularity} \label{sec_compacones}

In this section, we will consider bounded sub-and supersolutions of the equation 
	\[
	\Delta^N_\infty u = g(u) \qquad \text{on } \, \Omega \Subset M,
	\]
where $g$ is a function whose restriction to $[u_*,u^*]$ is non-decreasing and continuous, and $u_* = \inf_\Omega u, \ u^* = \sup_\Omega u$.
%
 	
 For given $b \ge 0$, consider a solution $\eta_{b}$ of 
	\begin{equation}\label{def_eta}
	\left\{ \begin{array}{ll}
	\eta_b''(t) = g(\eta_b(t)) \qquad \text{on a maximal interval } \, [0,T), \\[0.2cm]
	\eta_b(0) = u_*, \quad \eta_b'(0) = b.
	\end{array}\right.
	\end{equation}
Multiplying the equation by $2\eta'$ and integrating we deduce
	\begin{equation}\label{eq_etaprimo}
	[\eta_{b}'(t)]^2 - b^2 = G\big(\eta_{b}(t)\big), \qquad \text{where } \, G(s) = 2\int_{u_*}^s g(\sigma) \di \sigma.
	\end{equation}
If 
\begin{eqnarray}\label{b_cond_G}
b> \sqrt{\max\{-G_*,0\}},
\end{eqnarray}
where $G_* \doteq \inf_{[u_*,u^*]} G$, then $\eta'_b>0$ and a second integration shows that $\eta_{b}$ is implicitly defined by the identity
	\begin{equation}\label{eq_implicit}
	t = \int_{u_*}^{\eta_{b}(t)} \frac{\di s}{\sqrt{b^2 + G(s)}} \qquad \text{on } \, [0,T).
	\end{equation}
In particular, note that the family $\{\eta_b\}$ is increasing in $b$, whenever it is valued on $[u_*,u^*]$.

Given $a \in [u_*,u^*]$ we define
	\[
	R_b(a) \doteq \inf \Big\{ t \in [0,T) \ : \ \eta_{b}(t) \ge a\Big\}. 
	\] 
This constant encompasses the non translational invariance character of the inhomogeneous equation, and it helps us to deduce ``how far'' the $g$-cones can be defined. In view of \eqref{eq_etaprimo}, for any values $u_* \leq a_1< a_2 \leq u^*$ we have 
	\begin{equation}\label{grad_esti}
	\|\eta'_{b}\|_{L^\infty(R_b(a_1),R_{b}(a_2))} \le \sqrt{ b^2 + 2\int_{a_1}^{a_2} g_+} \le \sqrt{ b^2 + 2\int_{u_*}^{u^*} g_+}.
	\end{equation}

\begin{remark}\emph{
We recall that when the function $g$ is constant, let us say $g \equiv c$ for some $c \in \R$, the solutions of \eqref{def_eta} are the quadratic functions $\eta_b(t) = u_* + bt + \frac{c}{2}t^2$ considered in \cite{peres_schramm_sheffield_wilson,lu_wang_cpde,armstrong_smart_tams,mebrate_mohammed}.
}
\end{remark}

\begin{remark}\label{rmk_b=0}\emph{
If $g\geq 0$ on $[u_*,u^*]$, we will also consider the limit case of \eqref{eq_implicit} for $b=0$. Under the validity of the Keller-Osserman condition 
	\begin{equation}\label{eq_KO}\tag{KO}
	\int_{u_*^+} \frac{\di s}{\sqrt{G(s)}} < \infty,
	\end{equation}
uniqueness for \eqref{def_eta} does not hold, and we select $\eta_{0}$ as being the one defined by the limit identity
	\[
	t = \int_{u_*}^{\eta_{0}(t)} \frac{\di s}{\sqrt{G(s)}} \qquad \text{on } \, [0,T).
	\]
If \eqref{eq_KO} fails, necessarily $g(u_*) = 0$ and the only solution of \eqref{def_eta} with $b=0$ is the function $\eta_{0} \equiv u_*$. In this case, we set $R_0(a) \doteq  + \infty$ for every $a \in (u_*,u^*]$.
}
\end{remark}

For $z \in M$ fixed, we define the \emph{forward and backward $g$-cones} centered at $z$ as being, respectively, 
	\[
	\begin{array}{lcll}
	C^+_{z,b}(w) & = & \eta_{b}\big(\di(z,w) + R_b(u(z)) \big) & \quad \text{on } \, \mathcal{B}_{z}^+\big(R_b(u^*)- R_b(u(z))\big), \\[0.3cm]
	C^-_{z,b}(w) & = & \eta_b \big( R_b(u(z)) - \di(w,z) \big) & \quad \text{on } \, \mathcal{B}_{z}^-\big(R_b(u(z))\big).
	\end{array}
	\]
	
\begin{example}
\emph{For instance, if $g = 0$, 
	\[
	C^+_{z,b}(w) = u(z) + b \di(z,w), \qquad C^-_{z,b}(w) = u(z) - b \di(w,z) 
	\]
are the standard forward and backward cones. If $g \equiv c \not=0$, then 
	\[
	\begin{array}{lcll}
C^+_{z,b}(w) & = & u(z) + \big(b + cR_b(u(z))\big)\di(z,w) + \frac{c}{2}\di(z,w)^2 & \quad \text{on } \, \mathcal{B}_{z}^+\big(R_b(u^*)- R_b(u(z))\big), \\[0.3cm]
	C^-_{z,b}(w) & = & u(z) - \big(b + cR_b(u(z))\big)\di(w,z) + \frac{c}{2}\di(w,z)^2 & \quad \text{on } \, \mathcal{B}_{z}^-\big(R_b(u(z))\big).
	\end{array}
	\]
}
\end{example}
	
Since $\eta_b'>0$ on $(0,R_b(u^*))$, because of Corollary \ref{cor_radial_inf}, $C^+_{z,b}$ and $C^-_{z,b}$ satisfy, respectively,
	\[
	\left\{\begin{array}{ll}
	\disp \Delta_\infty C^+_{z,b} \le g(C^+_{z,b}) & \quad \text{on } \, \mathcal{B}_{z}^+\big(R_b(u^*)- R_b(u(z))\big)\backslash \{z\}, \\[0.2cm]
	C^+_{z,b}(z) = u(z), \\[0.2cm]
	C^+_{z,b} = u^* & \quad \text{on } \, \mathcal{S}_{z}^+\big(R_b(u^*)- R_b(u(z))\big),
	\end{array} \right.
	\]	
and
	\[
	\left\{\begin{array}{ll}
	\disp \Delta_\infty C^-_{z,b} \ge g(C^-_{z,b}) & \quad \text{on } \, \mathcal{B}_{z}^-\big(R_b(u(z))\big)\backslash \{z\}, \\[0.2cm]
	C^-_{z,b}(z) = u(z), \\[0.2cm]
	C^-_{z,b} = u_* & \quad \text{on } \, \mathcal{S}_{z}^-\big(R_b(u(z))\big). 
	\end{array} \right.
	\]		
Extend $C^+_{z,b}$ and $C^-_{z,b}$ outside of the respective domains by setting them equal to, respectively, $u^*$ and $u_*$, and call the resulting extensions $\bar C^+_{z,b}$ and $\bar C^-_{z,b}$. Note that the extensions are Lipschitz continuous on the entire $M$, and in view of \eqref{grad_esti} they satisfy 
	\begin{equation}\label{eq_Lipcone}
	\lip( \bar C^+_{z,b}, M) \le \sqrt{ b^2 + 2\int_{u_*}^{u^*} g_+(s)\di s }, \qquad \lip( \bar C^-_{z,b}, M) \le \sqrt{ b^2 + 2\int_{u_*}^{u^*} g_+(s)\di s }.
	\end{equation}

Our next result extends the celebrated comparison with cones theorem (cf. \cite{crandall_evans_gariepy,champion_depascale,lu_wang_cpde,mebrate_mohammed} and references therein) for $g$-cones. 

\begin{theorem}\label{comp_g-cones}
Let $\Omega \subset M$ be a bounded open set.
\begin{enumerate}
\item[i)] Suppose that $u \in \USC(\overline{\Omega}) \cap L^\infty(\Omega)$ satisfies
\begin{equation}\label{hyp_thm_cone_comp_1} 
\Delta^N_{\infty} u \ge g(u) \quad \text{in} \ \ \Omega, 
\end{equation}
and assume
\begin{equation*}\label{g_comp_cone_+}
g \in C(u(\overline{\Omega})) \ \text{be non-decreasing, and $b$ satisfy \eqref{b_cond_G}}.
\end{equation*}
Then, for any relatively compact, open set $K \subset \Omega$, and any forward $g$-cone $\bar C^+_{z,b}$ centered at $z \in \Omega\backslash K$, we have 
	\[ 
	u \le \bar C^+_{z,b} \qquad \text{on } \, \partial K \qquad \Longrightarrow \qquad u \le \bar C^+_{z,b} \qquad \text{on } \, \overline{K}.
	\]
\item[ii)] Suppose that $v \in \LSC(\overline{\Omega})\cap L^\infty(\Omega)$ satisfies
\begin{equation}\label{hyp_thm_cone_comp_2}
\Delta^N_{\infty} v \leq g(v) \quad \text{in} \ \ \Omega,
\end{equation}
and assume
\begin{equation*}\label{g_comp_cone_-}
g \in C(v(\overline{\Omega})) \ \text{be non-decreasing, and $b$ satisfy \eqref{b_cond_G}}.
\end{equation*}
Then, for any relatively compact, open set $K \subset \Omega$ and any backward $g$-cone $\bar C^-_{z,b}$ centered at $z \in \Omega\backslash K$, we have
	\[ 
	v \ge \bar C^-_{z,b} \qquad \text{on } \, \partial K \qquad \Longrightarrow \qquad v \ge \bar C^-_{z,b} \qquad \text{on } \, \overline{K}.
	\]
\end{enumerate}
\end{theorem}
\begin{proof}
The argument follows the standard comparison strategy. For i), we argue by contradiction and assume that $\gamma : = \max_{\overline{K}}(u- \bar C^+_{z,b}) > 0$. For $\eps >0$ small enough we define 
	\[
	\phi_\eps(t) = \eta_b(t + R_b(u(z))) - \frac{\eps}{2}t^2,
	\]	
and set $\varrho^+(x) = \di(z,x)$. Up to reducing $\eps$, we can assume that 
	\begin{equation}\label{eq_eps}
	\begin{array}{l}
	\gamma_\eps \doteq \max_{\overline{K}}(u-\phi_\eps(\varrho^+)) > \max \left\{ \frac{\gamma}{2}, \max_{\partial K}(u-\phi_\eps(\varrho^+))\right\}, \\[0.3cm]
	\phi_\eps' > 0 \quad \text{on } \, [0, R_b(u^*)],
	\end{array}
	\end{equation}
where the second line follows from the strict inequality in \eqref{b_cond_G}. Let $x_0 \in \mathrm{Int}(K)$ realize $\gamma_\eps$, and note that $\phi_\eps(\varrho^+) < u^*$ in a sufficiently small neighbourhood of $x_0$. Choose $\varrho_\eps^+ \succ_{x_0} \varrho^+$ as in Lemma \ref{lem_calabi}, and reduce $\eps$ to satisfy $\eps(\rho^+_\eps)^2 < \gamma$. By construction, $\gamma_\eps+ \phi_\eps(\varrho^+_\eps) \succ_{x_0} u$ and therefore, at the point $x_0$,
	\[
	g\big( \frac{\gamma}{2} + \phi_\eps(\varrho^+_\eps)\big) \le g\big( \gamma_\eps+ \phi_\eps(\varrho^+_\eps)\big) \le \Delta_\infty^{N,-} \big(\gamma_\eps+ \phi_\eps(\varrho^+_\eps)\big)
	\]
On the other hand, by Lemma \ref{lem_calabi}
	\[
	\Delta_\infty^{N,-} \big(\gamma_\eps+ \phi_\eps(\varrho^+_\eps)\big) = \phi_\eps''(\varrho_\eps^+) = g\big( \phi_\eps(\varrho_\eps^+) + \frac{\eps}{2} (\varrho_\eps^+)^2 \big) - \eps < g\big( \phi_\eps(\varrho_\eps^+) + \frac{\gamma}{2} \big),  
	\]
yielding to a contradiction. Case $ii)$ follows similarly.
\end{proof}

When $g$ is constant, with the same argument we deduce the following comparison with quadratic cones, well-known in the Riemannian setting (cf. \cite{lu_wang,mebrate_mohammed}), and a related local Lipschitz regularity result. For $z \in \Omega$ we set 
$$ \di^+(z) \doteq \sup \big\lbrace r>0 : \mathcal{B}^+_z(r) \Subset \Omega \big\rbrace, \qquad \di^-(z) \doteq \sup \big\lbrace r>0 : \mathcal{B}^-_z(r) \Subset \Omega \big\rbrace,$$
and
	\[
	\delta^+_\Omega (z) \doteq \max\big\{\di (z, w) : w \in \overline{\Omega}\big\}, \qquad \delta^-_\Omega (z) \doteq \max\big\{\di(w,z) : z \in \overline{\Omega}\big\}.
	\]
	
\begin{corollary}\label{thm_cone_comp}
Let $\Omega \subset M$ be a bounded open set, and let $c \in \mathbb{R}$.  
\begin{enumerate}
\item[i)] Suppose $u \in \USC(\Omega)\cap L^\infty(\Omega)$ solves
\begin{equation*} 
\Delta^N_{\infty} u \ge c \quad \text{in} \ \ \Omega. 
\end{equation*}
Then, for any relatively compact, open set $K \subset \Omega$, and any forward quadratic cone $C^+_{z,b}$ centered at $z \in \Omega\backslash K$, and $b + cR_b(u(z)) \geq c_- \delta^+_K(z)$, we have 
	\[ 
	\max_{\overline{K}}\left(u-C^+_{z,b}\right) = \max_{\partial K}\left(u-C^+_{z,b}\right).
	\]
Moreover, for every $r \in (0,\di^+(z))$ and every $w \in \mathcal{B}_z^+(r)$ it holds
\begin{equation}\label{ineq_lip_u}	
	\frac{u(w)-u(z)}{\di(z,w)} \le \max\left\{ c_- r, \frac{c_-}{2}r + \sup_{\xi \in \mathcal{S}_z^+(r)} \frac{u(\xi)-u(z)}{r} \right\} + \frac{c_-}{2} \di(z,w).
\end{equation}
\item[ii)] Suppose $v \in \LSC(\Omega)$ satisfies
\begin{equation*}
\Delta^N_{\infty} v \leq c \quad \text{in} \ \ \Omega.
\end{equation*}
For any relatively compact, open set $K \subset \Omega$ and any backward quadratic cone $C^-_{z,b}$ centered at $z \in \Omega\backslash K$, and $b + cR_b(u(z)) \geq c_+ \delta^-_K(z)$, we have
\begin{equation}\label{ineq_lip_v}
	\min_{\overline{K}} \left(v-C^-_{z,b}\right) = \min_{\partial K}\left(v-C^-_{z,b}\right).
\end{equation}
Moreover, for every $r \in (0,\di^-(z))$ and every $w \in \mathcal{B}_z^-(r)$ it holds
	\[	
	\frac{v(z)-v(w)}{\di(w,z)} \le \max\left\{ c_+ r, \frac{c_+}{2}r + \sup_{\xi \in \mathcal{S}_z^-(r)} \frac{v(z)-v(\xi)}{r} \right\} + \frac{c_+}{2} \di(w,z).
	\]	
\end{enumerate}
In particular, $u$ and $v$ are locally Lipschitz.
\end{corollary}
\begin{proof}
To prove \eqref{ineq_lip_u} and \eqref{ineq_lip_v} we just compare $u$ and $v$ with the cones
	\[
	C_{z,b}^+(w) = u(z) + (b+R_b(u(z)))\di(z,w) + \frac{c}{2}\di(z,w)^2, 
	\]
and 
	\[
	C_{z,b}^-(w) = u(z) - (b+R_b(u(z)))\di(w,z) + \frac{c}{2}\di(w,z)^2, 
	\]
either on $K$ or, respectively, on the balls $\mathcal{B}_z^+(r)$ and $\mathcal{B}_z^-(r)$. The restrictions $b + cR_b(u(z)) \geq c_- \delta^+_K(z)$ and $b + cR_b(u(z)) \geq c_+ \delta^-_K(z)$ enable us to apply Corollary \ref{cor_radial_inf} on the entire $K$.

\end{proof}

\begin{remark}\emph{
Corollary \ref{thm_cone_comp} shall be compared with Theorems 4.1 and 4.7 in \cite{mebrate_mohammed}. We remark that our quadratic cones are parametrized in a different way.}
\end{remark}

This comparison with cones theory allows us to assert the validity of the following strong finite maximum principle which will be crucial in the proof of our main results.

\begin{corollary}\label{cor_smp}
Let $\Omega \subset M$ be a connected open subset. If $u \in \USC(\Omega)$ is a subsolution of $\Delta^N_{\infty} u = 0$ in $\Omega$, then $u$ cannot attain an interior maximum point, unless $u$ is constant. If $v \in \LSC(\Omega)$ is a supersolution of $\Delta^N_{\infty} v = 0$ in $\Omega$, then $v$ cannot attain a interior minimum point, unless $v$ is constant.
\end{corollary}
\begin{proof}
We only describe the proof for subsolutions, since the other case follows along similar lines. Let $y \in \Omega$ be a maximum point, fix a forward ball $\mathcal{B}^+_{y}(r) \subset \Omega$ and $\alpha > 1$ as in \eqref{reversible_metric_ineq} for $U =\mathcal{B}^+_{y}(r)$. Let $z \in \mathcal{B}^+_{y}(\alpha^{-1}r/2)$, and note that the triangle inequality and \eqref{reversible_metric_ineq} imply $y \in \mathcal{B}^+_{z}(r/2) \subset \mathcal{B}^+_{y}(r)$. Applying Corollary \ref{compa_inftyLaplacian} on $\mathcal{B}^+_{z}(r)\backslash \{z\}$ to $u$ and the forward linear cone 
$$C^+_z(w) = u(z) + \frac{2(u(y)-u(z))}{r}\di (z,w),$$ 
we conclude that
$$ 0 \leq \Big(u(y)-u(z)\Big)\left(\frac{r}{2} - \di (z,y)\right) \leq 0,$$
hence $u$ is constant on $\mathcal{B}^+_{y}(r)$, and the conclusion follows by an open-closed argument. 
\end{proof}


Another important consequence of Corollary \ref{thm_cone_comp} is the following comparison theorem for the homogeneous case. Its proof, for Euclidean space with its flat Riemannian metric, was first given by Jensen \cite{jensen} with a delicate procedure (see also \cite{aronsson_crandall_juutinen, barles_busca}). A subsequent short and elegant argument has been provided by Armstrong and Smart \cite{armstrong_smart}, and in Appendix I below we describe the necessary changes to adapt their proof to the Finsler setting. 

\begin{theorem}\label{compa_inftyLaplacian}
Let $\Omega \Subset M$ and assume that $u \in \USC(\overline\Omega), v \in \LSC(\overline\Omega)$ satisfy
	\[
	\Delta^N_\infty u \ge 0, \quad \text{and} \quad \Delta^N_\infty v \le 0 \quad \text{in the viscosity sense on} \ \Omega.
	\]
Then,
	\[ 
	\max_{\overline{\Omega}}(u-v) = \max_{\partial \Omega}(u-v).
	\]
\end{theorem}

Comparison with standard linear cones is fundamental in the theory of the $\infty$-Laplace equation, and provides the bridge to show the equivalence between $\infty$-harmonicity and the absolutely minimizing Lipschitz property (see \cite{aronsson_crandall_juutinen, champion_depascale, crandall_visit}, and references therein).
\begin{definition}
Let $\Omega$ be a proper subset of $M$. We say that $u \in \lip(\Omega)$ is an absolutely minimizing Lipschitz function on $\Omega$ if, for all open subset $A\subset \Omega$,  
$$\lip(u,A) = \lip(u,\partial A). $$
\end{definition}

As recalled in the introduction, a characterization of $\Delta^N_\infty u = g(u)$ in terms of certain absolutely minimizing properties seems still unavailable. In order to achieve a uniform, global Lipschitz regularity without using the completeness of $M$, we introduce the following

\begin{definition}\label{def_slidingslope}
Given $\Omega \subset M$, $u \in C(\overline{\Omega})$ and a compact subset $A \subset \overline{\Omega}$, we define the sliding slope 
	\[
	b_{A} \doteq \inf \Big\{ b > \sqrt{\max\lbrace -G_*,0\rbrace} \ : \ \forall \, z \in A, \ \bar C^-_{z,b} \le u \le \bar C^+_{z,b} \ \text{on } \, A \Big\}.
	\]
If the set is empty, we define $b_A \doteq +\infty$. 
	\end{definition}
	
It is easy to see that $b_A < +\infty$ if and only if $u_{|A}$ is Lipschitz. 
	
\begin{example}
If $g = 0$, since $C^+_{z,b}(w) = u(z) + b \di(z,w)$ and $C^-_{z,b}(w) = u(z) - b \di(w,z)$ we have $b_A = \lip(u, A)$.
\end{example}

\begin{remark}\label{rem_useful}
\emph{If $g(u(\overline{\Omega})) \ge 0$, the convexity of $\eta$ solving \eqref{def_eta} implies that the set 
	\[
	\Big\{ b > 0 \ : \ \forall \, z \in A, \ \bar C^-_{z,b} \le u \le \bar C^+_{z,b} \ \text{on } \, A \Big\}
	\]
is the half-line $(b_A, \infty)$.
}
\end{remark}

\begin{lemma}
If $g(u(\overline{\Omega})) \geq 0$ then	
	\[
	b_{A} \le \lip(u, A).
	\]
\end{lemma}
\begin{proof}
	Let $b \doteq \lip(u,A)$, so upward linear cones $L^+_{z,b} = u(z) + b\di(z, \cdot)$ and downward linear cones $L^-_{z,b} = u(z) - b\di(\cdot, z)$ can be slid along $z \in A$ remaining, respectively, above and below the graph of $u$ on $A$. Since $\eta$ is convex up until it reaches value $u^*$, a forward $g$-cone $\bar C^+_{z,b}$ lies above $L^+_{z,b}$ up until the latter reaches the value $u^*$, hence $\bar C^+_{z,b} \ge u$ on $A$. Again by the convexity of $\eta$, a downward $g$-cone $\bar C^-_{z,b}$ with vertex at $z \in A$ and slope $b$ lies below the linear cone $L^-_{z,b}$ until the latter reaches value $u_*$, hence $\bar C^-_{z,b} \le u$ on $A$. By its very definition, $b_A \le b$.
\end{proof}

We will state now our main result of this section, Theorem \ref{teo_fundamental_intro}, in the following strengthened form:

\begin{theorem}\label{teo_fundamental}
Let $\Omega \Subset M$, and let $u \in C(\overline{\Omega})$ satisfy
	\[
	\Delta^N_\infty u = g(u) \qquad \text{on } \, \Omega,
	\]
where $g(u(\overline{\Omega})) \geq 0$. If $u$ is Lipschitz on $\partial \Omega$, then $u \in \lip(\overline{\Omega})$ and 
	\begin{equation*}\label{eq_globalLip}
	\lip(u, \Omega) \le \sqrt{ b_{\partial \Omega}^2 + 2\int_{u_*}^{u^*} g(s)\di s }.
	\end{equation*}
In particular, 
	\[
	\lip(u, \Omega) \le \sqrt{ \lip(u, \partial \Omega)^2 + 2\int_{u_*}^{u^*} g(s)\di s }.
	\]
\end{theorem}	
	
\begin{proof}
Pick $b > b_{\partial \Omega}$ and set for convenience 
	\[
	L_b = \sqrt{ b^2 + 2\int_{u_*}^{u^*} g(s)\di s }.
	\]
For $x,y \in \overline \Omega$, it is sufficient to show that
	\[
	u(x) \le u(y) + L_b \di(y,x),
	\]
since the thesis follows by letting $b \downarrow b_{\partial \Omega}$. By Remark \ref{rem_useful}, 
	\[
	\forall \, z \in \partial \Omega, \quad \bar C_{z,b} \le u \le \bar C^+_{z,b} \ \ \text{ on } \, \partial \Omega, 
	\]
thus comparison with $g$-cones implies $\bar C^-_{z,b} \le u \le \bar C^+_{z,b}$ on $\overline\Omega$, that is, 
	\[
	\bar C^-_{z,b}(w) \le  u(w) \le \bar C^+_{z,b}(w) \qquad \text{for every } \, w \in \overline \Omega, \ z \in \partial \Omega.	
	\]
If $y \in \partial \Omega$, then setting $z=y$, $w=x$ and using \eqref{grad_esti} we get
	\[
	u(x) \le \disp \bar C_{y,b}^+(x) \le \disp \bar C_{y,b}^+(y) + L_b \di(y,x) = u(y) + L_b \di(y,x).
	\]
On the other hand, if $x \in \partial \Omega$ and $y \in \overline{\Omega}$, setting $z = x$ and $w=y$ we deduce
	\[
	u(y) \ge \disp \bar C_{x,b}^-(y) \ge \disp \bar C_{x,b}^-(x) - L_b \di(y,x) = u(x) - L_b \di(y,x).
	\]
It remains to investigate the case $x,y \in \Omega$. Choose 
	\[
	b' = \inf \Big\{ h \geq 0 \ : \ u \ge \bar C^-_{x,h} \ \text{ on } \, \partial \Omega \Big\}.
	\]
Since $\Delta^N_\infty u \ge 0$ on $\Omega$, $u \in \lip_\loc(\Omega)$. In particular, the set defining $b'$ is non-empty, thus $b' < \infty$ and, by a compactness argument together with Remark \ref{rmk_b=0}, $b'$ is attained. The compactness of $\partial \Omega$, and the fact that $\bar C^-_{x,k} \ge \bar C^-_{x,h}$ if $k \le h$, guarantee the existence of $z_0 \in \partial \Omega$ such that $\bar C^-_{x,b'}(z_0) = u(z_0)$ and $C^-_{x,b'}(z) \le u(z)$ for every $z \in \partial \Omega$. Therefore, by comparison
	\[
	\bar C^-_{x, b'} \le u \qquad \text{on } \, \overline\Omega.
	\]
We examine the cone $\bar C^+_{z_0,b}$. Since it lies above the graph of $u$, hence above $C^-_{z,b'}$, its initial slope at $z_0$ must be, at least, the slope   of the solution $\eta_{u_*,b'}$ of the ODE corresponding to $C^-_{x,b'}$ at the point $R_{b'}(u(z_0))$. The latter is not smaller than the slope $b'$ (because $\eta_{u_*,b'}$ is convex), therefore we infer the inequality
	\[
	b \ge b'.
	\]
By comparison, $u \ge \bar C^-_{x,b'}$ on $\Omega$, implying
	\[
	\begin{array}{lcl}
	u(y) & \ge &  u(x) - \lip( \bar C^-_{x,b'}, M	) \di(y,x) \\[0.2cm]
	 & \ge & u(x) - L_{b'}\di(y,x) \ge u(x) - L_{b}\di(y,x).
	 \end{array}
	 \]
This concludes the proof. 	
	\end{proof}

%


\section{Proof of Theorem \ref{teo_main}}\label{sec_proof}

When the ``some/every" alternative occurs in $3),4),6),7),8)$, we will always assume the weaker and prove the stronger. For instance, when considering implication $2) \Rightarrow 4)$, we will show the validity of $4)$ for \emph{every} choice of $g$ as in the statement. On the other hand, in implication $4) \Rightarrow 1)$, for instance, we will only assume the validity of $4)$ for \emph{some} choice of $g$. In what follows, we set $u^* = \sup_M u$ and $u_* = \inf_M u$.\\[0.2cm]
\noindent ${\bf 1) \Rightarrow 2)}$.\\
Suppose, by contradiction, that there exists a solution $u$ of $\Delta_\infty^N u \ge 0$ on $M$ with sublinear growth $u(x) = o( \varrho^+(x))$ as $\varrho^+(x) \ra \infty$. Fix a compact set $K$. In view of the strong maximum principle, $u_K : = \max_{K} u < u^*$. Because of Corollary \ref{cor_radial_inf}, for every $\eps>0$ the function $w_\eps :=u_K + \eps \varrho_+$ satisfies $\Delta_\infty^N w_\eps \le 0$. Furthermore, our growth requirement on $u$ implies that $u < w_\eps$ outside of a relatively compact, open set $U$. The comparison theorem in Appendix I on $U \backslash K$ yields to 
	\[
	u \le w_\eps = u_K + \eps w_\eps \qquad \text{on $U \backslash K$, hence on $M \backslash K$,} 
	\] 
and letting $\eps \to 0$ we infer $u \le u_K$ on $M$, contradiction.\\[0.2cm]
${\bf 2) \Rightarrow 3)}$ is obvious, for every choice of such $g$.\\[0.2cm]
${\bf 2) \Rightarrow 4)}$.\\ 
By contradiction, assume that there exist $g \in C(\R)$, and $u$ satisfying 
	\[
	\left\{ \begin{array}{l}
	\Delta_\infty^N u \ge g(u) \ge 0 \qquad \text{on } \, \Omega, \\[0.2cm] 
	\sup_\Omega u < + \infty
	\end{array}\right. \qquad \text{with} \qquad \sup_\Omega u  > \sup_{\partial \Omega} u.
	\]	
Note that $u \in \lip_\loc(\Omega)$ because of Corollary \ref{thm_cone_comp}, so choosing $\gamma \in (\sup_{\partial \Omega} u, \sup_\Omega u)$ the function 
	\[
	v : = \left\{ \begin{array}{ll}
	\max\{\gamma, u\} & \quad \text{on } \, \Omega, \\[0.2cm]
	\gamma & \quad \text{on } \, M \backslash \Omega
	\end{array}\right.	
	\]
is bounded, non-constant and coincides with $\gamma$ in a neighbourhood of $\partial \Omega$, thus $\Delta_\infty^N v \ge 0$ on $M$ by Proposition \ref{elementary_prop_sub}. This contradicts $2)$.\\[0.2cm]
%
${\bf 3) \Rightarrow 1)}$ and ${\bf 4) \Rightarrow 1)}$.\\
We prove both of the implications with the same strategy, and split the proof only at the last step. Assume that either $3)$ or $4)$ holds for some choice of $g$. First, we redefine $g$ on an interval, say $[0,1]$ as follows: $g(t) \equiv g(1)$ for $t \ge 1$ and $g(t) = 0$ for $t \le 0$. In this way, the validity of $3)$ and $4)$ restricts to functions $u$ valued in $[0,1]$. Next, set
	\[
	\bar g(t) = \sup_{s \le t} g(s).
	\]
Then, $\bar g \in C(\R)$, $\bar g \ge g$, $\bar g(0) = 0$ and $\bar g$ is non-decreasing. Therefore, the validity of $3)$ or $4)$ for $g$ (and $u \in [0,1]$) implies its validity for $\bar g$, under the same restriction on $u$. Hence, up to replacing $g$ with $\bar g$, we can assume that $g$ be non-decreasing. Fix a point $x \in M$ and a small, forward regular ball $\mathcal{B}$ centered at $x$. Consider a smooth exhaustion $\{\Omega_j\} \uparrow M$ with $\mathcal{B} \Subset \Omega_j$ for each $j$. Set $A_j \doteq \Omega_j \backslash \overline{\mathcal{B}}$, and let $u_j$ be a solution of 
\begin{equation}
\left\{ \begin{array}{ll}
\Delta_{\infty} u_j = g(u_j) & \text{ on} \ A_j, \\[0.2cm]
u_j = f_j & \text{ on} \ \partial A_j,
\end{array}\right.
\end{equation}
where $f_j = 0$ on $\partial \mathcal{B}$ and $f_j = 1$ on $\partial \Omega_j$ (its existence follows from Perron method, using $0$ as a subsolution and $1$ as a supersolution, and is proved in Appendix II; note that $0 \le u_j \le 1$). Theorem \ref{teo_fundamental} guarantees that 
	\[
	\lip(u_j,A_j) \le \sqrt{ b^2_{\partial A_j} + 2\int_0^1 g(s)\di s}, 
	\]
With $b_{\partial A_j}$ the sliding slope of $\partial A_j$. We claim that $\{b_{\partial A_j}\}$ is decreasing, hence uniformly bounded, as $j \ra \infty$. Indeed, since $\partial \mathcal{B}$ separates $M$ and $\Omega_{j} \Subset \Omega_{j+1}$, every curve from $x \in \partial \mathcal{B}$ to a point $y \in \partial \Omega_{j+1}$ must cross $\partial \Omega_j$. Therefore,
	\[
	\di(\partial \mathcal{B}, \partial \Omega_{j+1}) \ge \di(\partial\mathcal{B}, \partial \Omega_j),
	\]
and thus any forward $g$-cone $\bar C^+_{x,b}$ that lies above $1$ on $\partial \Omega_j$ (i.e., it satisfies $R_b(1) \le \di( \partial \mathcal{B}, \partial \Omega_j)$) also lies above $1$ on $\partial \Omega_{j+1}$. Similarly, to every backward $g$-cone $\bar C^-_{y,b}$ that can be slid along $y \in \partial \Omega_j$ remaining below $0$ on $\partial \mathcal{B}$, the cones $\bar C^-_{z,b}$ centered at $z \in \partial \Omega_{j+1}$ and with the same $b$ remain below $0$ on $\partial \mathcal{B}$. This suffices to conclude $b^2_{\partial A_{j+1}} \le b^2_{\partial A_j}$. Therefore, $\{u_j\}$ is equi-Lipschitz, say with constant $L$. Extend $u_j$ with values $0$ on $\mathcal{B}$ and $1$ outside of $\Omega_j$. Up to subsequences, $\{u_j\}$ converges locally uniformly to a Lipschitz limit $u_\infty \ge 0$. By Proposition \ref{elementary_prop_sub}, $u_\infty$ satisfies $\Delta_\infty u_\infty = g(u_\infty)$ and $u_\infty = 0$ on $\partial \mathcal{B}$. We now exploit our assumptions. If $4)$ holds, applying the principle to $u_\infty$ on $\Omega = M \backslash \overline{\mathcal{B}}$ we deduce $u_\infty \equiv 0$. On the other hand, if $3)$ holds, first extend $u_\infty$ with $u_\infty \doteq 0$ on $\mathcal{B}$, and note that the resulting extension solves $\Delta_\infty u_\infty \ge g(u_\infty)$ on $M$. Apply then $3)$ to conclude that $u_\infty$ is constant, hence $u_\infty \equiv 0$. To show the forward completeness of $M$, pick a unit speed geodesic $\gamma : [0,T) \ra M$ issuing from the center $o$ of $\mathcal{B}$, and assume by contradiction that $T<+\infty$. Consider the functions $w_j = u_j \circ \gamma$, and note that $w_j =1$ after some $T_j<T$. From 
	\[
	\frac{ w_j(t)- w_j(s)}{t-s} \le \frac{u_j(\gamma(t)) - u_j(\gamma(s))}{\di(\gamma(s),\gamma(t))} \le \textrm{Lip}(u_j, M) \le L \qquad \forall \, 0 < s < t< T,
	\]
letting $t \ra T^-$ we deduce
	\[
	1 - w_j(s) \le L(T-s).
	\]
However, $w_j \ra 0$ locally uniformly, a contradiction if $s$ is chosen to be close enough to $T$.\\[0.2cm]
${\bf 5) \Rightarrow 2)}$ is obvious, with the choice $g(u) = \lambda u_+^\theta$.\\[0.2cm] 
${\bf 1) \Rightarrow 5)}$.\\
The argument follows the ideas in \cite{araujo_leitao_teixeira}. Let $\varrho^+$ be the forward distance from $o \in M$. For each $r>0$ we define the function $v_r$ on $\mathcal{B}^+_o(r) \subset M$ by 
$v_r(x) = \eta(\varrho^+(x))$, with
$$ \eta(t) = \tau(\lambda,\theta)\left[t - r + \left(\frac{\sup_{\partial \mathcal{B}^+_o(r)} u}{\tau(\lambda,\theta)}\right)^{\frac{1-\theta}{2}}\right]_{+}^{\frac{2}{1-\theta}},$$
and 
$$ 
\tau(\lambda,\theta) = \sqrt[1-\theta]{\frac{\lambda(1-\theta)^2}{2(1+\theta)}}.
$$
Note that $\eta \in C^2(\R)$ since $\theta \in (0,1)$. Using Corollary \ref{cor_radial_inf}, $v_r$ satisfies
	\[
	\left\{ \begin{array}{rll}
	\Delta_\infty^N v_r \leq \lambda (v_r)_{+}^{\theta} &  \text{on } \, \mathcal{B}^+_o(r) & \text{in the barrier sense,} \\[0.2cm]
	v_r = \sup_{\partial \mathcal{B}^+_o(r)} u, & \text{on } \partial \mathcal{B}^+_o(r). &
	\end{array}\right.
	\]
Since $u \leq v_r$ on $\partial \mathcal{B}^+_o(r)$, and $u$ is a subsolution of the above problem (in viscosity sense), we claim that $u \le v_r$ on $\mathcal{B}^+_o(r)$. In fact, if $u-v_r$ has a positive maximum $c$ at $x \in \mathcal{B}^+_o(r)$, let $\varrho_\eps^+ \succ_x \varrho^+$ be an upper barrier for $\varrho^+$ guaranteed by Calabi's trick. Then, $\phi : = c + \eta(\varrho_\eps^+) \succ_x u$ and thus 
	\[
	\lambda\phi_+^\theta \le \Delta_\infty^{N,+} \phi = \eta''(\varrho_\eps^+) = \lambda \eta(\varrho_\eps^+)_+^\theta < \lambda \phi_+^\theta \qquad \text{at } \, x,
	\]
contradiction. Next, by the growth assumption on $u$, we can find $0< \delta < 1$ such that
	\[
	\sup_{\partial \mathcal{B}^+_o(r)} u \leq \delta \tau(\lambda,\theta) r^{\frac{2}{1-\theta}}.
	\]
Summarizing, we can write
$$ u(x) \leq \tau(\lambda,\theta)\left[\varrho^+(x) - \left(1 - \delta^{\frac{1-\theta}{2}}\right)r \right]_{+}^{\frac{2}{1-\theta}}.$$
Letting $r \to +\infty$ we deduce that $u \le 0$ on $M$. To conclude, we apply $1) \Rightarrow 3)$ to obtain that $u$ is constant.\\[0.2cm]
\noindent ${\bf 1) \Rightarrow 6)}$ and ${\bf 1) \Rightarrow 7)}$.\\
Let $K \Subset M$ be compact, fix $o \in M$, $\varrho^+(x) = \di(o,x)$ and choose $R$ large enough that $K \subset \mathcal{B}^+_{o}(R)$. For $r > R$, the functions
	\[
	u_r(x) = \min \left\{ -1 + \frac{R}{r}(\varrho_+ - R), 0 \right\} \in \mathscr{L}(K,M)
	\]
satisfy
	\[
	\lip(u_r,M) = \frac{R}{r}, \qquad F(\nabla u_r) \le \frac{R}{r} \quad \text{a.e. on } \, M,
	\]
so letting $r \ra \infty$ we deduce both $6)$ and $7)$.\\[0.2cm]
\noindent ${\bf 7) \Rightarrow 6)}$ for some compact $K$.\\
The implication follows from the inequality
	\[
	\lip(u,M) \le \| F(\nabla u)\|_\infty \qquad \forall \, x \in \lip(M).
	\]
Indeed, for every unit speed curve $\gamma : [0, \ell] \ra M$ joining $x$ to $y$, and for every $u \in C^1(M)$, integrating the inequality $\di u (\gamma') \le F^*(\di u)F(\gamma') = F(\nabla u) \le \|F(\nabla u)\|_\infty$ on $[0,\ell]$ we infer
	\[
	u(y) = u(x) + \int_0^\ell \di u(\gamma'(t))\di t \le u(x) + \|F(\nabla u)\|_\infty\ell.
	\]
Choosing $\ell$ such that $\ell = \di(x,y) + j^{-1}$, and letting $j \ra \infty$, we deduce $u(y) \le u(x) + \|F(\nabla u)\|_\infty \di(x,y)$. The case $u \in \lip(M)$ follows by approximation.\\[0.2cm]
${\bf 6) \Rightarrow 1)}$.\\
Fix a compact set $K \subset M$ and a sequence of functions $\bar u_j \in \lip_c(M)$ with $\lip(\bar u_j,M) \ra 0$ and $\bar u_j \le -1$ on $K$. Up to replacing $\bar u_j$ with $\max\{\bar u_j,1\}$, we can assume that $-1 \le \bar u_j \le 0$ on $M$ and $\bar u_j = -1$ on $K$. By Ascoli-Arzel\'a theorem, up to subsequences, $\bar u_j \ra \bar u_\infty$ locally uniformly, for some $\bar u_\infty \in \lip(M)$, and from $\lip(\bar u_\infty,M) \le \liminf_j \lip(\bar u_j, M) = 0$ we deduce that $\bar u_\infty = -1$ on $M$. Now, the proof concludes exactly as the one for $3) \Rightarrow 1)$, up to defining $u_j = \bar u_j + 1$. \\[0.2cm] 
${\bf 1) \Rightarrow 8)}$.\\
By contradiction, if $u$ is a subsolution of 
	\[
	G(u)-F(\nabla u) = 0 \quad \text{on } \, \Omega,
	\]
	and $\sup_{\partial \Omega} u < \sup_\Omega u < \infty$, the function 
	\[
	v(x) = \int^{u(x)}_{0} \frac{\di s}{G(s)}
	\]
would be a subsolution of 
	\[
	\left\{ \begin{array}{ll}
	1- F(\nabla v) = 0  \quad \text{on } \, \Omega, \\[0.2cm]
	v_0 \doteq \sup_{\partial \Omega} v < \sup_\Omega v < \infty.
	\end{array}\right.
	\]
Let $\varrho^+$ be the forward distance from a fixed origin, and set $w_\eps \doteq v_0 + \eps \varrho^+$ for $\eps \in (0,1)$. We claim that $v \le w_\eps$ on $\Omega$. Once this is shown, letting $\eps \ra 0$ we would have $v \le v_0$, which is absurd. Assume therefore that $U \doteq \{v > w_\eps\}$ be non-empty. Since $M$ is forward complete, $w_\eps(x) \ra +\infty$ as $x$ diverges, thus $U$ is relatively compact and does not meet $\partial \Omega$. Pick a point $x \in U$ where $u-w_\eps$ attains a (positive) maximum value $c$, and let $\varrho_\eps^+ \succ_x \varrho^+$ be a barrier at $x$. Then, $\phi \doteq v_0 + c + \eps \varrho_\eps^+$ would touch $v$ from above at $x$, that would imply $0 \ge 1 - F(\nabla \phi) = 1- \eps F(\nabla \varrho^+_\eps) = 1-\eps$, contradiction.\\[0.2cm]
%
%
%
%
${\bf 8) \Rightarrow 1)}$.\\
Let $0< G \in C(\R)$ such that $8)$ holds. We define
	\[
	\hat G(t) = \min_{[0,t]} G(s).
	\]
Then, $\hat G$ is non-increasing and positive on $\R^+$, and from $\hat G \le G$ on $\R^+$ we deduce that $8)$ still holds, with $\hat G$ replacing $G$, provided that $u$ be non-negative on $\Omega$. Summarizing, we can assume that $G$ is non-increasing on $\R^+$, up to restricting the validity of $5)$ to nonnegative $u$. Fix a small, regular forward ball $\mathcal{B} = \mathcal{B}^+_{x_0}(3\eps)$, denote with $\widetilde{\nabla}$ the gradient induced by the dual Finsler structure $\widetilde{F}$, and define
	\[
	\widetilde{G}(t) = G(-t).
	\]
We aim to prove the existence of a function satisfying
	\begin{equation}\label{proprie_w_eik}
	\left\{\begin{array}{l}
	w \in C(M \backslash \mathcal{B}), \quad w \le 0, \\[0.2cm]
	w(x) \to -\infty  \ \text{ as $x$ diverges }, \\[0.2cm]
	\text{$w$ is a viscosity subsolution of } \, \widetilde{F}(\widetilde{\nabla} w) - \widetilde{G}(w) =0 \, \text{ on } \, M \backslash \overline{\mathcal{B}}. 
	\end{array}\right.
	\end{equation}
Here, the writing $w(x) \to -\infty$ as $x$ diverges means that $w$ has compact upper level sets in $M \backslash \mathcal{B}$. Once this is shown, we conclude that $M$ must be forward complete as follows: set
	\[
	h(x)\doteq \int_0^{w(x)} \frac{\di s}{\widetilde{G}(s)},
	\]
then $h \le 0$ and, since $G$ is non-increasing, $h(x) \ra -\infty$ as $x$ diverges. Furthermore, $h$ is a viscosity subsolution of $\widetilde{F}(\widetilde{\nabla} h) - 1=0$ on $M \backslash \mathcal{B}$. By Proposition 4.3 in \cite{camilli_siconolfi}, $h$ is Lipschitz continuous in the pseudo-distance $\tilde{\di}$ induced by $\widetilde{F}$:
	\[
	h(y) \le h(x) + L \tilde{\di}(x,y) = h(x) + L \di(y,x) \qquad \forall \, x,y \in M \backslash \mathcal{B}.
	\]
for some constant $L>0$. Take a maximal, forward geodesic $\gamma: [0,T) \ra M$ issuing from $x_0$, and suppose by contradiction that $T< +\infty$. Define $v(t) \doteq h(\gamma(t))$ on $[3\eps, T)$. By assumption, $v(t) \ra -\infty$ as $t \ra T^-$. On the other hand, 
	\[
	v(t) \ge v(3\eps) - L\di(\gamma(3\eps), \gamma(t)) \ge v(3\eps) + L(3\eps-t),
	\]	
contradiction.\par
The idea to prove the existence of $w$ is inspired by \cite{maripessoa,marivaltorta}. Let $\Omega_j \uparrow M$ be an increasing exhaustion of $M$ by means of relatively compact open sets with smooth boundary, satisfying $\overline{\mathcal{B}} \Subset \Omega_1$. We will construct a sequence of functions $\{w_j\}$ such that
	\begin{equation}\label{proprie_wj_eik}
	\left\{\begin{array}{l}
	w_j \in C(M \backslash \mathcal{B}), \qquad w_j \le 0 \ \text{ on } \, M \backslash \mathcal{B}, \ w_j > -1/2 \quad \text{on } \, \partial \mathcal{B} \\[0.2cm]
	w_{j+1} \le w_j \qquad \text{on } \, M \backslash \overline{\mathcal{B}}, \\[0.2cm]
	\|w_{j+1} - w_j\|_{L^\infty(\Omega_j\backslash \mathcal{B})} < 2^{-j}, \\[0.2cm]
	w_j \equiv -j \qquad \text{outside of some compact set $C_j$}, \\[0.2cm]
	\text{$w_j$ is a viscosity subsolution of } \, \widetilde{F}(\widetilde{\nabla} w_j) - \widetilde{G}(w_j) = 0 \ \text{ on } \, M \backslash \overline{\mathcal{B}}. 
	\end{array}\right.
	\end{equation}
Once this is done, $\{w_j\}$ locally uniformly converges to some $w \in C(M \backslash \mathcal{B})$, and from $w \le w_j = -j$ outside of $C_j$ we deduce that $w(x) \ra -\infty$ as $x$ diverges. By stability of viscosity solutions, $w$ satisfies all of the properties in \eqref{proprie_w_eik}. Fix a sequence $\{\lambda_j\} \subset C(M)$ such that
	\[
	\begin{array}{l}
	0 \ge \lambda_j \ge -1, \quad \lambda_j = 0 \ \ \text{ on } \, \mathcal{B}, \quad \lambda_j \equiv -1 \ \ \text{ on } \, M \backslash \Omega_j, \\[0.2cm]
	\lambda_{j+1} \ge \lambda_j \ \ \text{ on } \, M, \quad \text{and} \quad \lambda_j \uparrow 0 \quad \text{locally uniformly on } \, M.
	\end{array}
	\]
We proceed inductively. Set $w_0 \equiv 0$ and define the forward balls $\mathcal{B}_1 = \mathcal{B}_{x_0}^+(\eps)$ and  $\mathcal{B}_2 = \mathcal{B}_{x_0}^+(2\eps)$, so that $\mathcal{B}_1 \Subset \mathcal{B}_2 \Subset \mathcal{B}$. Fix a smooth cutoff $\psi \in C^\infty_c(\mathcal{B})$ satisfying $\psi \equiv 1$ on $\mathcal{B}_2$, and denote with $\varrho^+(x) = \di(x_0,x)$ the forward distance to $x_0$ in $M$. For each $j$, define the Lipschitz function
	\[
	s_j(x) = j \cdot \max \left\{ \frac{\eps - \varrho^+}{\eps}, -1\right\}.
	\]	
Since $-\varrho^+(x)$ coincides with the signed backward distance to $x_0$ in $\widetilde{F}$, applying Corollary \ref{cor_radial_inf} to $(M, \widetilde{F})$ we deduce that $s_j$ is a viscosity subsolution of 
	\[
	\widetilde{F}(\widetilde{\nabla} s_j) - \widetilde{G}(s_j) - \frac{j}{\eps}\psi(x) = 0 \qquad \text{on } \, M \backslash \overline{\mathcal{B}_1}.	
	\]
We will construct $\{w_j\}$ in such a way that $w_j \ge s_j$ on $M$, in particular, $w_j = 0$ on $\partial \mathcal{B}_1$. This is trivial for $w_0$. Having fixed $w=w_j$, we define the obstacles $g_i = w + \lambda_i$ for $i>j$. For each $i$, we consider the following Perron class:
	\[
	\mathscr{F}[g_i] = \left\{ v \in C(\overline{\Omega_i \backslash \mathcal{B}_1}) \ : \ \begin{array}{l}
	v \le g_i, \ \text{ and $v$ is a viscosity subsolution of} \\[0.2cm]
	\widetilde{F}(\widetilde{\nabla} v) - \widetilde{G}(v) - \frac{j+1}{\eps} \psi(x) = 0 \ \text{ on } \, \Omega_i\backslash \overline{\mathcal{B}_1}
	\end{array}
	\right\},
	\]
and the envelope
	\[
	u_i(x) \doteq \sup \Big\{ v(x) : v \in \mathscr{F}[g_i]\Big\},
	\]
namely the solution of the obstacle problem on $\Omega_i \backslash \mathcal{B}_1$ with obstacle $g_i$. Perron class is non-empty, since it contains the constant $-j-1$. Furthermore, since $\lambda_i = 0$ on $\mathcal{B}$, we have $g_i \ge s_j + \lambda_i \ge s_{j+1}$, and from $\psi \equiv 0$ outside of $\mathcal{B}$ we deduce $s_{j+1} \in \mathscr{F}[g_i]$. This and $0 \ge u_i \ge s_{j+1}$ guarantee that $u_i=0$ on $\partial \mathcal{B}_1$. For $v \in \mathscr{F}[g_i]$, the function  
	\[
	h_v = \int_0^{v(x)} \frac{\di s}{\widetilde{G}(s)}
	\]
is a subsolution of 
	\[
	\widetilde{F}(\widetilde{\nabla} h_v) - 1 - \frac{j+1}{\eps} \cdot \frac{1}{\inf_{[-j-1,0]} \widetilde{G}} = 0
	\] 
on $M \backslash \mathcal{B}_1$. Proposition 4.3 in \cite{camilli_siconolfi} guarantees that $h_v$ is Lipschitz with constant $L_j$ only depending on $j$. Thus, functions $v \in \mathcal{F}[g_i]$ with $v \ge -j-1$ are equiLipschitz, in particular $u_i \in \lip(\Omega_i \backslash \mathcal{B}_1)$. By stability, $u_i$ is still a viscosity subsolution of
	\[
	\widetilde{F}(\widetilde{\nabla} u_i) - \widetilde{G}(u_i) - \frac{j+1}{\eps} \psi(x) = 0	\qquad \text{on } \, \Omega_i \backslash \mathcal{B}_1,
	\]
and in fact it is also a viscosity supersolution of the same equation on the open set $\{ u_i < g_i\}$. For $i$ large enough to satisfy $C_j \Subset \Omega_i$,   
	\[
	-j-1 \le u_i \le g_i = -j-1 \qquad \text{on } \, \Omega_i \backslash C_j.
	\]	
Thus, $u_i = -j-1$ in a neighbourhood of $\partial \Omega_i$. Extending $u_i$ with $-j-1$ outside of $\Omega_i$ produces a subsolution (still named $u_i$) of 
	\[
	\widetilde{F}(\widetilde{\nabla} u_i) - \widetilde{G}(u_i) = 0 \quad \text{on } \, M \backslash \mathcal{B}.
	\]
Clearly, by construction $u_i \in \mathscr{F}[g_{i'}]$ for every $i' > i$. Therefore, the sequence $\{u_i\}$ is monotone increasing and equiLipschitz, and hence converges to a limit function $u \in \lip(M \backslash \mathcal{B}_1)$ that vanishes on $\partial \mathcal{B}_1$. \par
\noindent{\bf Claim:} $u \equiv w$.\\
We first prove that $u \ge -j$ on $M \backslash \mathcal{B}_1$. We proceed by contradiction, assuming that the open set $U = \{u < -j- \delta\}$ be non-empty for some $\delta>0$. Note that $U$ might intersect $\mathcal{B}$, where the term $\psi$ does not vanish, but $\overline{U} \subset M \backslash \overline{B_1}$ since $u = 0$ on $\partial B_1$. Choose $i_0$ large enough that 
	\[
	U_{i_0} = \{ u < g_{i_0} - \delta \} \neq \emptyset. 
	\]
This is possible since $g_i \uparrow w$ locally uniformly. By monotonicity, $u_i < g_i - \delta$ on $U_{i_0}$ for every $i \ge i_0$, meaning that the solution of the obstacle problem $u_i$ detaches from the obstacle $g_i$ on $U_{i_0}$. Therefore, $u_i$ is also a supersolution of  
	\[
	\widetilde{F}(\widetilde{\nabla} u_i) - \widetilde{G}(u_i) - \frac{j+1}{\eps} \psi(x) = 0 \qquad \text{on } \, U_{i_0}
	\]
and, by stability, so is $u$ on $U_{i_0}$. From $U = \bigcup_{i_0} U_{i_0}$, we deduce that $u$ is a supersolution of 
	\[
	\widetilde{F}(\widetilde{\nabla} u) - \widetilde{G}(u) - \frac{j+1}{\eps} \psi(x) = 0 \qquad \text{on } \, U,
	\]
and, as a consequence, a supersolution of $\widetilde{F}(\widetilde{\nabla} u) - \widetilde{G}(u) = 0$ on $U$. At this stage, we use property $8)$ to $v : = -u$, that is a subsolution of 
	\[
	G(\nabla v) - F(\nabla v) = 0 \qquad \text{on } \, U,
	\]
to deduce that $\sup_U v = \sup_{\partial U} v$, contradicting the very definition of $U$ and proving the claim. Next, fix $i_0$ with $C_j \Subset \Omega_{i_0}$, and $\delta > 0$ small. From $u \ge -j$ and $w= -j$ on $M \backslash C_j$, we deduce that $u_i \uparrow -j$ uniformly on $\partial \Omega_{i_0}$. Choose $i>>i_0$ such that  
	\[
	u_i > -j - \frac{\delta}{2} \qquad \text{on } \, \partial \Omega_{i_0}. 
	\]
It follows that the function
	\[
	v_i = \left\{ \begin{array}{ll}
	\max\{ w-\delta, u_i \} & \quad \text{on } \, \Omega_{i_0}, \\[0.2cm]
	u_i & \quad \text{on } \, \Omega_i \backslash \Omega_{i_0}
	\end{array}\right.
	\]
belongs to $\mathscr{F}[g_i]$, and therefore $u_i \ge v_i$ on $\Omega_i$ by the maximality of $u_i$. In particular, $u_i \ge w-\delta$ holds on $\Omega_{i_0}\backslash \mathcal{B}$ for $i$ large enough. By the arbitrariness of $i_0$ and $\delta$, this proves that $u_i \uparrow w$ locally uniformly on $M \backslash \mathcal{B}$, hence $u \equiv w$.\\[0.2cm]
	To conclude, from $u_i \uparrow u \equiv w$ locally uniformly we can choose $i$ large enough such that, setting $w_{j+1} \doteq u_i$, $w_{j+1}$ satisfies all of the requirements in \eqref{proprie_wj_eik}. \\[0.2cm]
${\bf 1) \Rightarrow 9)}$.\\
As stated in the introduction, the proof of Ekeland principle given in \cite[p.444]{ekeland}, see also \cite[p.85]{amr}, does not use the symmetry of $\di$, and can therefore be repeated verbatim.\\[0.2cm]
${\bf 9) \Rightarrow 1)}$.\\
The argument is due to \cite{weston,sullivan}, and we reproduce it here for the sake of completeness. Let $\{x_j\}$ be a forward Cauchy sequence, and define the function
	\[
	f \ : \ M \ra [-\infty,0], \qquad f(x) = - \limsup_j \di(x,x_j).
	\]
The goal is to prove the existence of $\bar x \in M$ such that $f(\bar x) = 0$.	Fix $\eps > 0$ and $j_\eps$ guaranteed by the Cauchy condition. From
	\[
	\di(x_{j_\eps}, x_j) < \eps \qquad \forall \, j > j_\eps
	\]	 
we deduce $f(x_{j_\eps}) \ge -\eps$, hence $\sup_M f = 0$. Furthermore, the triangle inequality implies $f(y) \le f(x) + \di(x,y)$, hence $f$ is locally Lipschitz and finite everywhere. Fix $\delta \in (0,1)$ and, by $9)$, let $\bar x$ satisfy 
	\[
	f(\bar x) \ge -\delta, \qquad f(y) \le f(\bar x) + \delta \di(\bar x, y).
	\]
Choosing $y = x_j$ for $j > j_\eps$ we deduce
	\[
	-\eps \le f(\bar x) + \delta \di(\bar x, x_j). 	
	\]
Thus, letting $j \ra \infty$ along a sequence realizing $f(\bar x)$, and then letting $\eps \ra 0$, we get
	\[
	0 \le f(\bar x) - \delta f(\bar x) = (1-\delta)f(\bar x) \le 0, 
	\]
and we conclude $f(\bar x) = 0$.

\section{Appendix I: A homogeneous comparison} \label{sec_appe_I}

\begin{theorem}
Let $\Omega \Subset M$ and assume that $u \in \USC(\overline\Omega), v \in \LSC(\overline\Omega)$ are bounded on $\Omega$ and satisfy
	\[
	\Delta_\infty^N u \ge 0, \quad \text{and} \quad \Delta_\infty^N v \le 0 \quad  \text{in the viscosity sense on} \ \Omega.
	\]
Then,
	\[ 
	\max_{\overline{\Omega}}(u-v) = \max_{\partial \Omega}(u-v).
	\]
\end{theorem}
\begin{proof}[Proof: sketch] Since the Finsler structure is non-symmetric, we need to adapt some notation from \cite{armstrong_smart}. First of all, by a compactness argument, we fix 
$\alpha >1$ satisfying \eqref{reversible_metric_ineq} on the whole of $\Omega$. For any $\varepsilon >0$ and $\Omega \Subset M$ let us denote 
$$\Omega^+_\varepsilon = \{ x \in \Omega : \overline{\mathcal{B}^+_{x}}(\varepsilon) \subset \Omega\}, \quad \text{and} \quad \Omega^-_\varepsilon = \{ x \in \Omega : \overline{\mathcal{B}^-_{x}}(\varepsilon) \subset \Omega\}.$$
We set $\Omega_\varepsilon \doteq \Omega^-_\varepsilon \cap \Omega^+_\varepsilon$. Up to reducing $\varepsilon$, we will assume that ${B_{x}^+(2 \eps)}$ and ${B_{x}^-(2\eps)}$ are relatively compact for all $x \in \Omega$.

For $x \in \Omega^+_\varepsilon$ and $y \in \Omega^-_\varepsilon$, define
	\[ 
	u^\varepsilon(x) \doteq \max_{\overline{\mathcal{B}^+_x}(\varepsilon)} u \qquad \text{and} \qquad  v_\varepsilon(y) \doteq \min_{\overline{\mathcal{B}^-_{y}}(\varepsilon)} v.
	\]
As in \cite{armstrong_smart}, applying Corollary \ref{thm_cone_comp} we can prove that $u^{\varepsilon}$ and $v_{\varepsilon}$ are solutions of the following finite difference inequalities 
\begin{eqnarray}\label{finite_dif_ineq}
S^{-}_\varepsilon u^{\varepsilon} (x) - S^+_\varepsilon u^{\varepsilon} (x) \ \leq \ 0 \ \leq \ S^{-}_\varepsilon v_{\varepsilon} (x) - S^+_\varepsilon v_{\varepsilon} (x) 
\end{eqnarray}
for every $x \in \Omega^+_{2\alpha\varepsilon}$,
where $S^\varepsilon$ and $S_\varepsilon$ are defined as follows
	\[ 
	S^+_\varepsilon u(x)  \doteq \max_{y \in \overline{\mathcal{B}^+_{x}}(\varepsilon)} \frac{u(y) - u(x)}{\varepsilon}, \quad \text{and} \quad S^-_\varepsilon u(x) \doteq \max_{y \in \overline{\mathcal{B}^-_{x}}(\varepsilon)} \frac{u(x) - u(y)}{\varepsilon}.
	\]

Now, arguing as in \cite[Lem 4]{armstrong_smart} we can conclude that
$$ \sup_{\Omega^+_{\alpha\varepsilon}} \left(u^{\varepsilon} - v_{\varepsilon}\right) = \sup_{\Omega^+_{\alpha\varepsilon} \backslash \Omega^+_{2\alpha\varepsilon}} \left(u^{\varepsilon} - v_{\varepsilon}\right).$$
The conclusion then follows by passing to the limit $\varepsilon \rightarrow 0$.
\end{proof}

\section{Appendix II: The Dirichlet problem}

Let $\Omega \subset M$ be relatively compact, and let $g : \R \times T^*\overline{\Omega} \ra \R$ with the following properties:
	\begin{equation}\label{ipo_g_app}
	\begin{array}{ll}
	(i) & \quad g\in C(\R \times T^*\overline{\Omega}), \\[0.2cm]
	(ii) & \quad \sup_{(t,v) \in I \times T^*\overline{\Omega}} |g| < \infty \qquad \text{for every compact } \, I \subset \R.
	\end{array}
	\end{equation}
	 
\begin{theorem}\label{existence_dirichlet_linear}
Let $g$ satisfying \eqref{ipo_g_app}, and let $u_1,u_2 \in C(\overline\Omega)$ solving
	\[
	\left\{ \begin{array}{rl}
	\Delta_\infty^N u_1 \ge g(u_1,\di u_1) & \text{on } \ \Omega, \\[0.2cm] 
	\Delta_\infty^N u_2 \le g(u_2,\di u_2) & \text{on } \ \Omega, \\[0.2cm] 
	u_1 \le u_2 & \text{on } \ \overline{\Omega}.
	\end{array}\right.
	\]
Then, for every $\zeta \in C(\partial \Omega)$ with $u_1 \le \zeta \le u_2$, there exists $u \in C(\overline\Omega)$ such that 
	\[
	\left\{ \begin{array}{rl}
	\Delta_\infty^N u = g(u,\di u) &  \text{on } \ \Omega, \\[0.2cm] 
	u_1 \le u \le u_2 &  \text{on } \ \overline{\Omega}, \\[0.2cm]
	u = \zeta &  \text{on } \ \partial \Omega.
	\end{array}\right.
	\]	
\end{theorem}

\begin{remark}
\emph{Note that the above existence result does not need any comparison theorem.
}
\end{remark}

\begin{proof}
We will employ the Perron method. Fix $I = [ \min_{\overline{\Omega}} u_1, \max_{\overline\Omega}u_2]$ and choose $c \in \R^+$ such that 
	\begin{equation}\label{cond_c_perron}
	c > \max_{T\overline\Omega \times I} |g|.
	\end{equation}
Consider the Perron class
	\[
	\mathscr{P} = \Big\{ v \in C(\overline\Omega) \ : \ u_1 \le v \le u_2, \ \Delta_\infty^N v \ge g(v,\di v), \ v \le \zeta \ \text{ on } \, \partial \Omega\Big\},
	\]
and the Perron envelope $u = \sup\{ v : v \in \mathscr{P}\}$ on $\Omega$. By \eqref{cond_c_perron}, $-c \le \Delta_\infty^N v \le c$ for every $v \in \mathscr{P}$. Because of Corollary \ref{thm_cone_comp}, $\mathscr{P}$ is uniformly locally Lipschitz continuous, hence $u \in \lip_\loc(\Omega)$. Given $x \in \partial \Omega$ and $\delta>0$, let $\eps>0$ small enough that the signed distance $\varrho^-(y) = -\di(y,x)$ is smooth on $\mathcal{B}_x^-(\eps)\backslash \{x\}$ and that 
	\[
	\begin{array}{l}
	u_2 > \zeta(x) - \delta \quad \text{on } \, \overline{B_{x}^-(\eps) \cap \Omega}, \qquad \zeta > \zeta(x) - \delta \quad \text{on } \, \overline{B_{x}^-(\eps) \cap \partial \Omega}, \\[0.2cm]
	u_1 < \zeta(x) + \delta \quad \text{on } \, \overline{B_{x}^+(\eps) \cap \Omega}, \qquad \zeta < \zeta(x) + \delta \quad \text{on } \, \overline{B_{x}^+(\eps) \cap \partial \Omega}.
	\end{array}
	\]
	Set $\zeta^-_\delta(x) \doteq \zeta(x) - \delta$, and let $b >> 1$ large enough in such a way that the backward quadratic cone
	\[
	C^-_{b,x}(y) \doteq \zeta_\delta(x) -(b + R_b(\zeta^-_\delta(x)))\di(y,x) + \frac{c}{2} \di(y,x)^2, 
	\] 
defined on $\mathcal{B}_x^-(\eps)$, satisfies $C^-_{b,x} < u_1$ on $ \mathcal{S}_x^-(\eps) \cap \overline\Omega$. By Corollary \ref{thm_cone_comp} we then have 
$C_{b,x}^- \le u_2$ on $\overline{\mathcal{B}_x^-(\eps) \cap \Omega}$, and 
	\[
	\Delta_\infty^N C_{b,x}^- \ge c\ge g(C_{b,x}^-, \di C_{b,x}^-) \qquad \text{on } \, \mathcal{B}_x^-(\eps) \cap \{ C_{b,x}^- > u_1\}.
	\]
It follows that 	
	\[
	w : = \left\{ \begin{array}{ll}
	\max\{ C_{b,x}^- , u_1\} &  \text{on } \, \overline{\mathcal{B}_x^-(\eps) \cap \Omega},\\[0.2cm]
	u_1 &  \text{otherwise} 
	\end{array}\right.	
	\]
lies in $\mathscr{P}$ and therefore
	\begin{equation}\label{liminf_u}
	\liminf_{y \ra x} u(y) \ge \liminf_{y \ra x} w(y) \ge \liminf_{y \ra x} C_{b,x}^-(y) = \zeta(x)-\delta.
	\end{equation}
Similarly, setting $\zeta^+_\delta(x) = \zeta(x) + \delta$, we consider the forward quadratic cone
 	\[ 
	C^+_{b,x}(x) \doteq \zeta_\delta^+(x) + (b + R_b(\zeta_\delta^+(x))) \di(x,y) - \frac{c}{2} \di(x,y)^2 
	\] 
that for large enough $b$ solves
	\[
	\left\{ \begin{array}{rl}
	\Delta_\infty^N C^+_{b,x} \le -c &  \text{on } \, \mathcal{B}_x^+(\eps), \\[0.2cm]
	C^+_{b,x} > u_2 &  \text{on } \, \mathcal{S}_x^+(\eps) \cap \overline\Omega.
	\end{array}\right.
	\]  
We claim that $v < C^+_{b,x}$ on $\mathcal{B}_x^+(\eps)$ for every $v \in \mathscr{P}$. Indeed, this holds by construction on $\mathcal{S}_x^+(\eps) \cap \overline\Omega$, while on $\partial \Omega \cap \overline{\mathcal{B}_x^+(\eps)}$ we have
	\[
	v \le \zeta < \zeta^+_\delta(x) \le C^+_{b,x},
	\]
thus $v < C^+_{b,x}$ on $\partial( \mathcal{B}_x^+(\eps) \cap \Omega)$. If $v - C^+_{b,x}$ attains a non-negative maximum $m_0$ at $x_0 \in \mathcal{B}_x^+(\eps) \cap \Omega$, then $C^+_{b,x} + m_0$ is a smooth function that touches $v$ from above and satisfies $\Delta_\infty^N C^+_{b,x}(x_0) \le -c < g(u(x_0), \di u(x_0))$, contradiction. Thus, $v \le C^+_{b,x}$ on $\mathcal{B}^+_x(\eps) \cap \Omega$ and, taking supremum, $u \le C^+_{b,x}$ there. Hence, 
	\[
	\limsup_{y \ra x} u(y) \le \limsup_{y \ra x} C^+_{b,x}(y) = \zeta(x)+\delta,
	\]
thus coupling with \eqref{liminf_u} and letting $\delta \ra 0$ we infer $u \in C(\overline\Omega)$ with $u=\zeta$ on $\partial \Omega$. By the stability of subsolutions with respect to uniform convergence (Proposition \ref{elementary_prop_sub}), $\Delta_\infty^N u \ge g(u,\di u)$ on $\Omega$. We are left to prove that $u$ is also a supersolution. Suppose, by contradiction, that there exist $x_0 \in \Omega$ and $\phi \prec_{x_0} u$ defined in a small, relatively compact neighbourhood $U \Subset \Omega$ of $x_0$ such that $\Delta_\infty^N \phi(x_0) > g(\phi,\di \phi)(x_0)$. If $u(x_0) = u_2(x_0)$, then $\phi \prec_{x_0} u_2$, contradicting the fact that $u_2$ is a supersolution. Therefore, $u(x_0) < u_2(x_0)$. Up to subtracting to $\phi$ a function $\psi \succ_{x_0} 0$ that is positive on $U \backslash \{x_0\}$ and vanishes at $x_0$ at second order, we can assume that $\phi < u$ on $U \backslash \{x_0\}$. By continuity of $\zeta$ and since $\phi$ is smooth, up to shrinking $U$ and choosing $\eps$ small we can satisfy any of the following properties: 
	\[
	\left\{ \begin{array}{ll}
	\phi + \eps < u & \quad \text{on } \, \partial U, \\[0.2cm]
	\phi + \eps \le u_2 & \quad \text{on } \, U, \\[0.2cm]
	\Delta_\infty^N \phi > g( \phi + \eps, \di \phi) & \quad \text{on } \, U.
	\end{array}\right.
	\]
It follows that 
	\[
	\hat u : = \left\{ \begin{array}{ll}
	\max\{ u, \phi + \eps\} & \quad \text{on } \, U, \\[0.2cm]
	u & \quad \text{on } \, \overline\Omega \backslash U
	\end{array}\right.
	\]
	lies in $\mathscr{P}$, and since $\hat u(x_0) > u(x_0)$ this contradicts the definition o $u$.	
\end{proof}

\bibliographystyle{amsplain}

\end{document}